\documentclass[a4paper,10pt]{article}
\usepackage{pdf14}
\usepackage[T1]{fontenc}
\usepackage[utf8]{inputenc}
\usepackage[l2tabu, orthodox]{nag}
\usepackage{geometry}
\usepackage[american]{babel}
\usepackage{caption}
\usepackage{subcaption}

\usepackage[tbtags]{amsmath}
\usepackage{amsthm,amssymb}
\usepackage{bm}

\usepackage{graphicx}
\usepackage{float}
\usepackage{wrapfig}

\usepackage{array}
\usepackage{booktabs}

\usepackage[hyphens]{url}
\usepackage[pdfa]{hyperref}
\hypersetup{
  colorlinks = true,
  allcolors = blue,
  pdfencoding = auto,
  psdextra
}

\usepackage{cleveref}
\usepackage{import}
\usepackage{bm}
\usepackage{subcaption}
\usepackage{interval}
\intervalconfig{soft open fences}

\newcommand{\orcidlogo}{\includegraphics[height=10pt]{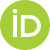}}
\newcommand{\orcid}[1]{\href{https://orcid.org/#1}{\orcidlogo}}

\theoremstyle{plain}
\newtheorem{theorem}{Theorem}[section]
\newtheorem{lemma}[theorem]{Lemma}
\newtheorem{corollary}[theorem]{Corollary}
\newtheorem{remark}[theorem]{Remark}

\newcommand{\R}{\mathbb{R}}
\newcommand{\M}{\mathcal{M}}
\newcommand{\N}{\mathbb{N}}
\newcommand{\C}{\mathcal{C}}
\newcommand{\W}{\mathcal{W}}

\renewcommand{\S}{\mathcal{S}}

\usepackage{stmaryrd}
\newcommand{\IF}[1]{\llbracket{#1}\rrbracket}

\newcommand{\Mat}[1]{\bm{#1}}

\newcommand{\eb}{{\bm{e}}}
\newcommand{\gb}{{\bm{g}}}
\newcommand{\nb}{{\bm{n}}}
\newcommand{\vb}{{\bm{v}}}
\newcommand{\wb}{{\bm{w}}}
\newcommand{\xb}{{\bm{x}}}
\newcommand{\yb}{{\bm{y}}}

\newcommand{\tangent}{\bm{\tau}}
\newcommand{\conormal}{\bm{\nu}}
\newcommand{\tconormal}{\widetilde{\conormal}}

\newcommand{\Pt}{\mathcal{P}}
\newcommand{\tPt}{\widetilde{\Pt}}

\newcommand{\norm}[1]{\lVert{#1}\rVert}
\newcommand{\Norm}[1]{\left\lVert{#1}\right\rVert}

\newcommand{\PP}{P_{\nb}} 
\newcommand{\PPh}{P_{\nb_h}} 

\newcommand{\ds}{\,\mathrm{d}s}

\newcommand{\dtau}{\,\mathrm{d}\tau}
\newcommand{\dA}{\,\mathrm{d}A}

\usepackage{colonequals}

\DeclareMathOperator{\arctantwo}{arctan2}

\newenvironment{subsubfigure}[2][]{%
  \begin{subfigure}[#1]{#2}%
    \stepcounter{subsubfigure}%
}{%
    \addtocounter{subfigure}{-1}%
  \end{subfigure}%
}
\newcounter{subsubfigure}

\numberwithin{equation}{section}

\begin{document}
\title{Surface Minkowski tensors to characterize shapes on curved surfaces}

\author{Lea Happel, Hanne Hardering, Simon Praetorius\thanks{Corresponding author: \href{mailto:simon.praetorius@tu-dresden.de}{simon.praetorius@tu-dresden.de}}, and Axel Voigt}

\maketitle

\begin{abstract}
    We introduce surface Minkowski tensors to characterize rotational symmetries of shapes embedded in curved surfaces. The definition is based on a modified vector transport of the shapes boundary co-normal into a reference point which accounts for the angular defect that a classical parallel transport would introduce. This modified transport can be easily implemented for general surfaces and differently defined embedded shapes, and the associated irreducible surface Minkowski tensors give rise to the classification of shapes by their normalized eigenvalues, which are introduced as shape measures following the flat-space analog. We analyze different approximations of the embedded shapes, their influence on the surface Minkowski tensors, and the stability to perturbations of the shape and the surface. The work concludes with a series of numerical experiments showing the applicability of the approach on various surfaces and shape representations and an application in biology in which the characterization of cells in a curved monolayer of cells is considered.

\end{abstract}

\section{Introduction}
Minkowski tensors (MT) are generalizations of Minkowski functionals (MF) and are powerful tools to characterize the shape of spatial structures with respect to rotational symmetries. They are based on a solid mathematical foundation, provided by integral and stochastic geometry, and are endowed with robustness and completeness theorems. This is at least true in flat space where these measures have been intensively used in various disciplines~\cite{SchroderTurk2011Minkowski}. More recently, the concept was extended to characterize embedded shapes in spherical surfaces~\cite{Chingangbam2017Tensor,Collischon2024Morphometry,Duque2024Geometric}. The main application in these studies comes from cosmology.

Our interest in characterizing rotational symmetries of embedded shapes in curved surfaces results from biology. Here the objects of interest are individual cells in a cell monolayer which defines the curved surface. The shapes of cells and their alignments with neighboring cells can determine orientational order in these layers which influences their mechanical properties~\cite{nitschke2025active}. Nematic order, characterized by rotational symmetry under $180^\circ$, has been widely studied in cell monolayers~\cite{duclos2017topological,saw2017topological,kawaguchi2017topological} and has been linked to cellular behaviors and tissue organization. Also more complex orders, such as tetratic order, characterized by rotational symmetry under $90^\circ$~\cite{Cislo_NaturePhysics_2023} and hexatic order, characterized by rotational symmetry under $60^\circ$~\cite{li2018role,armengol2023epithelia}, and even general $p$-atic orders, characterized by rotational symmetry under $2 \pi / p$ with $p$ being an integer~\cite{happel2025quantifying} have been observed in experimental systems. This offers new perspectives on how cells self-organize, respond to mechanical cues and trigger morphological changes in tissues. However, to identify these $p$-atic orders and use them to model morphogenetic processes first requires a robust and versatile tool to characterize the shape of cells on curved surfaces. In \Cref{fig:exp} two examples for curved cell monolayers in morphogenetic processes with cells as embedded shapes are shown. The cell boundaries are marked in gray and a strong variability in shape with respect to rotational symmetry of the cells and their orientation can be seen.

\begin{figure}
\centering
    \includegraphics[width=0.669\textwidth]{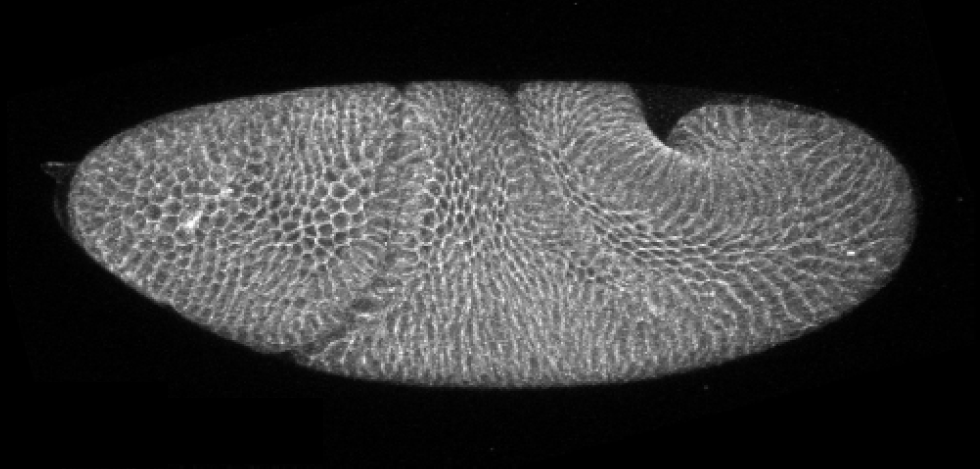}
    \includegraphics[width=0.319\textwidth]{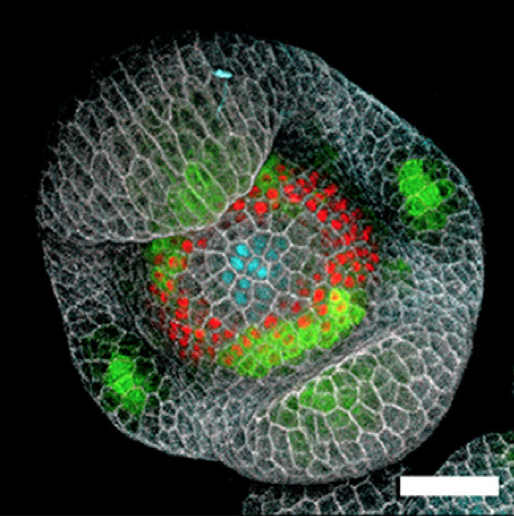}
    \caption{\label{fig:exp} Examples of curved cell monolayers in developmental biology.\ (left) Lateral view of wildtype Drosophila germband extension, with permission from~\cite{munster2019attachment} (Figure 4).\ (right) Confocal image of Arabidopsis flower bud, with permission from~\cite{PRUNET2016114} (Figure 1). In both images cell boundaries are shown in gray.}
\end{figure}

The work of~\cite{Chingangbam2017Tensor} provides a first step in the direction to characterize such shapes. It introduces the notion of surface MF and surface MT by translating the definitions in flat space to a curved manifold. Their definition is then applied to spherical geometries only, since intrinsic geometric knowledge about the surface, i.e., geodesics and parallel transport, needs to be employed. We will revise these definitions and propose a different definition for surface MT tensors that is verified on regular surfaces and proven to give the expected behavior on more general surfaces. Similar to the work of~\cite{Collischon2024Morphometry}, we focus on irreducible surface MT and their eigenvalues and eigenvectors as these provide a natural shape measure and anisotropy characterization.

We further develop and compare practical implementations on more general surfaces, including triangulated and piecewise parametrized surfaces. Also the embedded shape on the surface can be given differently, either as a parameterized curve or in an implicit description as a levelset of a function defined on the surface. We analyze the convergence properties of the approximations associated with these representations and numerically verify
these properties on various examples with increased complexity. This demonstrates robustness and a broad applicability.

The paper is structured as follows: In \Cref{s:2} we introduce the geometric setup and provide the definitions of surface MF and surface MT as well as irreducible surface MT.\@ To justify the definitions, we examine their properties in \Cref{s:3}. The surface MT can be efficiently evaluated numerically, using two types of approximations of the continuous definitions, by geodesic polygons and straight line polygons in the embedding, as introduced in \Cref{sec:approximation}.
All constructions and approximation results are validated by computations and some applications are provided in \Cref{sec:numerical-experiments}.

\section{Surface Minkowski functionals and surface Minkowski tensors}\label{s:2}

\subsection{Geometric setup}\label{sec:geometric-setup}
\begin{wrapfigure}{r}{0.25\textwidth}
\centering
\vspace{-2em}
    \includegraphics[width=\linewidth]{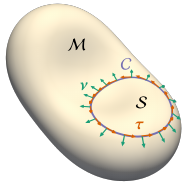}
    \caption{Illustration of the general naming of embedded shapes on surfaces.}\label{fig:2}
\vspace{-2em}
\end{wrapfigure}
Let $\M$ be a smooth two-dimensional orientable Riemannian manifold with bounded Gaussian curvature $K$.
Let $\S\subset\M$ denote a compact simply connected subdomain of $\M$ with $\int_\S K\dA < 2\pi$ and (piecewise) smooth boundary contour $\C=\partial \S$, see \Cref{fig:2} for an illustration.

We represent the curve $\C$ by a periodic arc-length parametrization $\gamma\colon[0,L]\to\M$, i.e., $\gamma(0)=\gamma(L)$, length of the curve $L=|\C|$, and constant speed $\norm{\dot{\gamma}}\equiv 1$. The tangent vector along the curve is denoted by $\tangent=\dot{\gamma}\in T_{\gamma}\M$ and the outer co-normal vector by $\conormal\in T_{\gamma}\M$.
The geodesic curvature $k_g$ along the curve $\C$ can be calculated by $k_g(s) = -\ddot{\gamma}(s)\cdot\conormal(s)$ for $s\in[0,L]$ the curve parameter.

If $\M$ is isometrically embedded in $\R^3$ we denote by $\nb\colon\M\to\R^3$ the outer normal vector field to $\M$ and assume that the tangent, co-normal and surface normal form a right-handed moving frame $(\conormal,\tangent,\nb)$.

\subsection{Surface Minkowski functionals}

Following the classical definition of integral moments of convex bodies in two dimensions, see, e.g.,~\cite{Mueller1953Ueber,Hadwiger1973Studien,SchroderTurk2009Tensorial}, we introduce surface MF as
\begin{align}
    \W_0 &\colonequals \int_{\S} 1\dA, &
    \W_1 &\colonequals \int_{\C} 1\ds, &
    \W_2 &\colonequals \int_{\C} k_g\ds,
\end{align}
that is, the area of the enclosed subdomain, curve length, and total curvature.
In flat space, $\W_2$ is equivalent to the Euler characteristic of $\S$. By Hadwiger's characterization scheme, any additive, continuous, and motion-invariant functional of $\S$ can be expressed as a linear combination of the MF.\@ Since we assume that $\S$ is simply connected, $\W_2$ would contain no additional information in flat space. However, this is different on curved surfaces. Here, the geodesic curvature is connected to the enclosed Gaussian curvature of the subdomain as described by the Gau\ss--Bonnet theorem.

\begin{remark}\label{rem:gauss-bonnet}
    Assume that the boundary curve $\C$ is smooth. Then, the classical Gau\ss--Bonnet theorem tells us that
    \[
      \int_{\S} K\dA + \int_{\C} k_g\ds = 2\pi\chi(\S)
    \]
    with $\chi(\S)$ the Euler characteristic of $\S$. We consider simply connected subdomains $\S\subset\M$ and thus $\chi(\S)=1$.
\end{remark}

Even in the flat space, the MF alone are not well suited to quantify characteristics of the shape of $\S$ that relate to directional anisotropy or orientation. This motivates the introduction of MT~\cite{SchroederTurk2013Minkowski}.

\subsection{Surface Minkowski tensors}

To make use of the directional information encoded in the co-normal field, MT (in $\R^2$) are defined as $W_1^p = \int_{\C} \conormal^{ \otimes p} \ds$~\cite{SchroederTurk2013Minkowski}, with $\conormal^{ \otimes p}= \conormal\otimes\ldots\otimes\conormal$ the $p$-times outer vector product. These tensors are translation invariant, and measure anisotropy with respect to different symmetry orders $p\in \N$.
In two dimensions, the anisotropy with respect to $p$ can be quantified by a scalar value $\mu_p$, see~\cite[there named $q_s$]{Klatt_JSM_2022}, which is a scaled eigenvalue of the trace free part of $W_1^{p}$ in the sense of~\cite{Qi_JSC_2005} (see~\cite[mainly eq. $(30)$]{Virga_EPJE_2015} for a detailed introduction for $p=3$).

In~\cite{Chingangbam2017Tensor,Collischon2024Morphometry} a generalization of the flat domain to surfaces is postulated, based on parallel transport of the co-normal vectors along the curve $\C$ into a fiducial point $\xb\in \M$.

We propose a different generalization that preserves the following properties of the MT in a flat domain:
\begin{enumerate}
    \item The shape measures $\mu_p$ deduced from the MT should be independent of the point of evaluation, and invariant under isometries.
    \item Reference shapes for perfect $p$-isotropies should correspond to regular geodesic polygons on the manifold.
\end{enumerate}
The construction using parallel transport will not yield these properties, as the canonical shape measure for a regular geodesic triangle on the sphere with ninety degree angles will correspond to the shape measure of a square in a flat domain. Moreover, for $p>1$ the construction is not independent of the fiducial point, as parallel transport of a vector around a closed curve results in an angular defect reflecting the curvature enclosed,~\cite{doCarmo1976Differential}.

Note, however, that any construction that has the second property will not be continuous as $\int_\S K\dA \to 2\pi$, as any regular geodesic $p$-polygon on the sphere with vertices approaching a great circle will have a constant shape measure $\mu_p=1$ that has to vanish once the circle is reached. This motivates the restriction $\int_\S K\dA < 2\pi$.

We denote by $\Pt^\gamma_{s\to t}\colon T_{\gamma(s)}\M\to T_{\gamma(t)}\M$ the parallel transport from $s\to t$ backwards along $\gamma$. For a vector field $\bm{w}(s) = \cos(\vartheta(s)) \conormal(s) + \sin(\vartheta(s)) \tangent(s)$ along the curve $\gamma$, the parallel transport reads
\begin{align}
    \Pt^\gamma_{s\to t}\bm{w}(s) &= \cos(\vartheta(s) + \phi(s,t))\conormal(t) + \sin(\vartheta(s) + \phi(s,t))\tangent(t),
\end{align}
with $\phi(s,t) = \int_{t}^{s} k_g(\tau)\dtau$.
The angle defect accumulation is denoted by
\begin{align}
    \eta(s,t) &\colonequals \frac{\int_{\S}K\dA}{\int_{\C}k_g\ds} \int_{t}^{s} k_g(\tau)\dtau
    = \left(\frac{2\pi}{\int_{\C} k_g\ds} - 1 \right)\int_{t}^{s} k_g(\tau)\dtau.
\end{align}
Denote by $R_t(\varphi)$ the rotation in the tangential plane $T_{\gamma(t)}\M$ by angle $\varphi$.
We define a defect corrected parallel transport by $\tPt^\gamma_{s\to t}\bm{w}(s) \colonequals R_t(-\eta(s,t))\Pt^\gamma_{s\to t}\bm{w}(s)$, given explicitly by
\begin{align}
    \tPt^\gamma_{s\to t}\bm{w}(s) &= \cos(\vartheta(s) + f(s,t))\conormal(t) + \sin(\vartheta(s) + f(s,t))\tangent(t),
\end{align}
where we write
\begin{align}
    f(s,t) &\colonequals\phi(s,t)+\eta(s,t) = \frac{2\pi}{\int_{\C}k_g\ds} \int_{t}^{s} k_g(\tau)\dtau.\label{eq:f}
\end{align}

Denoting by $\tconormal(s,t)\colonequals \tPt^\gamma_{s\to t}\conormal(s)$, we define the surface MT of rank $p$ on $\M$ associated to the fiducial point $\gamma(t)$,
\begin{align}
    \W_{1}^{p}(\gamma(t)) &\colonequals \int_{t}^{t+L}{\tconormal(s,t)}^{\otimes p}\ds, \label{eq:surface-minkowski-functional-1}
\end{align}
with $\otimes p$ denoting the $p$-times outer tensor product, $\tconormal^{\otimes p} = \tconormal\otimes\ldots\otimes\tconormal$. This definition still depends on the point $\gamma(t)$. However, this dependency does not affect the shape characterization, as will be shown in \cref{sec:point-independence}.

\begin{remark}
    For closed geodesic curves $\C$, we have $\int_\C k_g=0$ and thus $\int_S K\dA=2\pi$. While this case is excluded, the tensor definition could be extended by setting $\eta$ to zero. This coincides with the definition given in~\cite{Chingangbam2017Tensor}.
\end{remark}

\begin{remark}
    In flat space the surface MT definition coincides with the usual definition~\cite{Klatt_JSM_2022}. Other types of MT used in flat space, which are defined in terms of a position vector relative to a fixed coordinate system and also mixed position and normal vector tensor products, do not directly generalize to non-flat surfaces and are thus not adopted here.
\end{remark}

\subsection{Irreducible surface Minkowski tensors }\label{sec:irreducible-minkowski-tensors}
To analyze the shape of the curve $\C$ based on the surface MT~\eqref{eq:surface-minkowski-functional-1}, it is convenient to consider only terms of a decomposition of these tensors into their irreducible components. These components then contain information that is unique for the respective tensor rank and does not contain redundant information, cf.~\cite{Kapfer_Theses_2012,Mickel_JCP_2013,Klatt_JSM_2022,Collischon2024Morphometry}. The irreducible form of a symmetric tensor $\Mat{A}$ is denoted by $\IF{\Mat{A}}$ and represents its trace-free part, see~\cite{Virga_EPJE_2015,TothTuryshev2022Efficient}. In addition to reducing redundancy, a two-dimensional symmetric trace-free tensor of any rank $p \geq 1$ has only 2 independent components. In~\cite[Supplementary Information, section S2]{armengol2023epithelia} these two components are introduced and used to encode the eigenvalues and the eigenvector of the tensor~\cite{Qi_JSC_2005} as shape measures of the MT.\@ We follow this path and consider the irreducible surface MT of rank $p$
\begin{align}
    \IF{\W_1^p(\gamma(t))} &= \int_{t}^{t+L} \IF{\tconormal^{\otimes p}(s,t)}\ds.\label{eq:surface-minkowski-functional-if}
\end{align}
The two independent components of $\IF{\tconormal^{\otimes p}}\simeq\gb_p\colonequals(g_{p,1},g_{p,2})$ are given in terms of the angle of $\tconormal$ with the first local basis vector $\conormal$,
\begin{align}
    g_{p,1}(s,t) &= \cos(p f(s,t)), &
    g_{p,2}(s,t) &= \sin(p f(s,t)),\label{eq:independent-components-smooth}
\intertext{with $f$ as in~\eqref{eq:f}. The integral of these components can readily be computed as}
    \overline{g}_{p,1}(t) & =\int_{t}^{t+L} \cos(p f(s,t))\ds, &
    \overline{g}_{p,2}(t) & =\int_{t}^{t+L} \sin(p f(s,t))\ds. \label{eq:integrated-components-smooth}
\end{align}
Since $\IF{\W_1^p(\gamma(t))}$ is also symmetric and trace-free, its independent components are directly represented by the integrated $g_{p,i}$, i.e., $\IF{\W_1^p(\gamma(t))}\simeq\overline{\gb}_p\colonequals(\overline{g}_{p,1},\overline{g}_{p,2})$.

From the independent components the eigen-spectrum of the irreducible surface MT can be directly computed: A positive eigenvalue $\lambda_p^+(t)={\big(\overline{g}_{p,1}^2(t) + \overline{g}_{p,2}^2(t)\big)}^{1/2}$ is connected to the eigenvectors
\begin{align}\label{eq:minkowski-eigenvectors}
    \eb_{p,n}^{+}(t) &= \cos(\vartheta_{p,n}^+(t))\conormal(t) + \sin(\vartheta_{p,n}^+(t))\tangent(t) ,\quad n\in \{0,1,\ldots,p-1\},
\end{align}
with angles
\begin{align*}
    \vartheta_{p,n}^+(t)\colonequals\frac{2n\pi + \arctantwo(\overline{g}_{p,2}(t),\overline{g}_{p,1}(t))}{p}, \quad n\in \{0,1,\ldots,p-1\}.
\end{align*}
Likewise, a negative eigenvalue $\lambda_p^-(t)=-\lambda_p^+(t)$ is connected to the eigenvectors $\eb_{p,n}^{-}$ expressed in terms of their angle representation,
\begin{align*}
    \vartheta_{p,n}^-(t)\colonequals\frac{(2n+1)\pi + \arctantwo( \overline{g}_{p,2}(t), \overline{g}_{p,1}(t))}{p}, \quad n\in \{0,1,\ldots,p-1\}.
\end{align*}

\begin{remark}
    The quantities $\lambda_p$ are indeed eigenvalues of the irreducible surface Minkowski tensor $\IF{\W_1^p(\gamma(t))}$, up to the constant factor $2^{p-1}$, in the sense of generalized tensor $Z$-eigenvalues as defined in~\cite{Qi_JSC_2005}. For $p = 2$, this reduces to the classical eigenvalues of a symmetric trace-free matrix.
    Let $\nu$ be a unit vector and let $W = \IF{\nu^{\otimes p}}$ be a rank-$p$ symmetric trace-free tensor. An eigenvector $\eb$ corresponding to the eigenvalue $\lambda$ is then characterized as a solution of the system
    \[
    W \bullet^{p-1} \eb^{\otimes (p-1)} = \lambda \eb,
    \qquad \|\eb\| = 1,
    \]
    where $U \bullet^{p-1} V$ denotes the contraction of the last $p-1$ indices of $U$ with the first $p-1$ indices of $V$.
    If $\eb$ is an eigenvector, then all its rotations by $2\pi/p$ are also eigenvectors. This follows from the symmetric structure of the tensor. See~\cite{Virga_EPJE_2015} for an explicit construction in the case $p = 3$.
\end{remark}

A measure for the degree of rotational symmetry with respect to $p$ can be extracted from the magnitude of the eigenvalues $\lambda_p^\pm$~\cite{Klatt_JSM_2022,Collischon2024Morphometry} as
\begin{align}
    \mu_p = \frac{|\lambda_p^\pm|}{|\lambda_0|} \in [0,1].
\end{align}
The division by $\lambda_0=\W_1$ ensures scale invariance. In the following we will mainly discuss properties of the shape measures $\mu_p$, and will call them normalized eigenvalues of the irreducible surface MT.\@

\section{Properties of surface Minkowski tensors}\label{s:3}

We show the independence of the shape measures $\mu_p$ on the choice of the fiducial point, and discuss what this means for the eigenvectors of $\IF{\W_1^p(\gamma(t))}$. We also briefly discuss the analogs of translation and scale invariance of irreducible surface MT.\@
Moreover, we look at the special case of geodesic polygons, and see in what sense equal-angled polygons fulfill the role of perfect isotropic shapes for a given rank $p$.
Finally, we will observe a stability property of the irreducible surface MT for polygons with respect to small perturbations.

\subsection{Point independence}\label{sec:point-independence}
\begin{lemma}[Point independence of the eigenvalues]\label{lem:point-independence-eigenvalue}
    The eigenvalues $\lambda^{\pm}_p(t)$ of the irreducible surface MT $\IF{\W_1^p(\gamma(t))}$ are independent of the fiducial point.
\end{lemma}
\begin{proof}
    First note that the angle $f$, and thus the $\IF{\W_1^p(\gamma(t))}$, are independent of the parametrization of $\C$.
    In order to investigate the independence of the irreducible surface MT on the choice of fiducial point, we differentiate
    \begin{align*}
        \frac{d}{dt}\overline{g}_{p,1}(t) & = \frac{d}{dt} \int_{t}^{L+t} \cos(p f(s,t))\ds \\
        & = p m k_g(t)\overline{g}_{p,2}(t) + \cos(p m \int_{\C}k_g \ds)- 1\\
        & = p m k_g(t)\overline{g}_{p,2}(t),
    \end{align*}
    with $m=\frac{2\pi}{\int_{\C}k_g\ds}$.
    Analogously, we obtain $\frac{d}{dt}\overline{g}_{p,2}(t) = -p m k_g(t)\overline{g}_{p,1}(t)$.
    \begin{align*}
        \frac{d}{dt}{\lambda_p^{\pm}(t)}^2 &= 2\overline{g}_{p,1}(t)\frac{d}{dt}\overline{g}_{p,1}(t)+2\overline{g}_{p,2}(t)\frac{d}{dt}\overline{g}_{p,2}(t) = 0.\qedhere
    \end{align*}
\end{proof}

As the eigenvectors of $\IF{\W_1^p(t)}$ are defined only in the tangent space $T_{\gamma(t)}\M$, any notion of point independence necessarily involves a transport between tangent spaces. In our case, this transport is realized by the defect-corrected parallel transport.

\begin{lemma}[Point independence of the eigenvectors]\label{lem:point-independence-eigenvector}
    The eigenvectors $\eb_{p,n}^\pm(t)$ and $\eb_{p,n}^\pm(s)$ of the irreducible surface MT $\IF{\W_1^p(t)}$ and $\IF{\W_1^p(s)}$, respectively,
    fulfill
    \[\tPt^\gamma_{s\to t}\eb_{p,n}^\pm(s)=\eb_{p,n}^\pm(t)\quad \textrm{for all } t,s\in\rinterval{0}{L}.\]
\end{lemma}
\begin{proof}
    W.l.o.g.\ we consider only the case $\overline{g}_{p,1}(s)>0$, and consider the quotient $q(s)\colonequals \frac{\overline{g}_{p,2}(s)}{\overline{g}_{p,1}(s)}$. It holds
    \begin{align*}
        \frac{d}{ds}q(s)
        &= \frac{1}{\overline{g}_{p,1}^2(s)}\big( \frac{d}{ds}\overline{g}_{p,2}(s)\overline{g}_{p,1}(s)-\frac{d}{ds}\overline{g}_{p,1}(s)\overline{g}_{p,2}(s)\big)\\
        &= -p m k_g(s)\frac{1}{\overline{g}_{p,1}^2(s)} \big(\overline{g}_{p,2}^2(s)+\overline{g}_{p,1}^2(s)\big)\\
        &= -p m k_g(s) \big(q^2(s)+1\big).
    \end{align*}
    Thus, we have $\frac{d}{ds}\arctan(q(s)) = -p m k_g(s)$ and $\frac{d}{ds} \vartheta_{p,n}^\pm(s) = -m k_g(s) = -\partial_s f(s,t)$.
    This concludes the proof, as
    \begin{align*}
        \tPt^\gamma_{s\to t}\eb_{p,n}^{\pm}(s)
        &= \cos\left(\vartheta_{p,n}^{\pm}(s)+f(s,t)\right)\conormal(t) + \sin\left(\vartheta_{p,n}^{\pm}(s)+f(s,t)\right) \tangent(t) . \qedhere
    \end{align*}
\end{proof}

\begin{remark}
    From the point independence of the irreducible surface MT $\IF{\W_1^p}$ it follows that also the surface MT $\W_1^p$ are point independent, i.e., the tensors are constant w.r.t.\ the defect-corrected transport $\tPt^\gamma$. This is a consequence of the decomposition of the symmetric tensors in their irreducible components,
    cf.~\cite{Kapfer_Theses_2012,TothTuryshev2022Efficient}.
\end{remark}

\subsection{Invariance under embeddings, translations, and scaling}
Translation and scaling invariance are essential properties of the flat space MT.\@ For the generalization to manifolds~\eqref{eq:surface-minkowski-functional-1} similar properties hold, but must be seen in the context of the curved spaces. We will state them in this section. They are a consequence of the fact, that although $\M$ is assumed to be isometrically embedded into $\R^3$, the construction of the surface MT is intrinsic to the manifold $\M$.

We first note, that any metric preserving change of the manifold $\M$ as a submanifold of $\R^3$, like translation or rotation, does not change any intrinsic quantity used to defined the surface MT.\@

\begin{lemma}
    The surface MT are invariant under isometries of the manifold.
\end{lemma}

A generalization of translation in flat space to manifolds is a flow generated by Killing vector fields. These flows are continuous isometries of the manifold, i.e., are metric preserving and in particular preserve the angles used to define the surface MT.\@

\begin{lemma}[Translation invariance]
    The surface MT are invariant under any flow of the curve $\C$ generated by a Killing vector field on the manifold $\M$.
\end{lemma}

\begin{remark}
  On the sphere $\M=S^2$ Killing vector fields generate rotations about any axis. If $\M$ is a surface of revolution there is at least the rotation axis of the surface about which the curves can be rotated without changing their surface MT.\@ However, several manifolds do not posses any Killing vector fields.
\end{remark}

The property known as scaling invariance translates to invariance under constant conformal changes of the metric, i.e., $g\mapsto \lambda^2 g$ with $\lambda\in \R\backslash\{0\}$. While this changes the length of the curve, the angles $f$ defined by~\eqref{eq:f} are still preserved, and thus the normalized Minkowski eigenvalues of the irreducible surface MT are preserved:

\begin{lemma}[Scaling invariance]
    The normalized eigenvalues $\mu_p$ of the irreducible surface MT $\IF{\W_1^p}$ are invariant under constant conformal changes of the metric of the manifold.
\end{lemma}

Note that this does not hold true for general conformal changes, as for $\tilde{g}=\exp(2\varphi)g$, with a smooth function $\varphi:\M\to \R$, we have $\int_{C}\tilde{k}_g \tilde{\ds} = \int_{C}\exp(-\varphi)(k_g +g(\nabla \varphi,\nu)) \ds$, and a similar expression for only a section of a curve. Thus, we cannot expect invariance of the surface MT.\@

\subsection{Geodesic polygons}\label{sec:geodesic-polygons}

For piecewise smooth curves, the definition of integrated curvature is amended by the sum of turning angles. A special case are geodesic polygons, where the curvature along the smooth edges vanishes. This simplifies the definition of the independent components of the irreducible surface MT considerably.

\begin{wrapfigure}{r}{0.3\textwidth}
\centering
    \includegraphics[width=\linewidth]{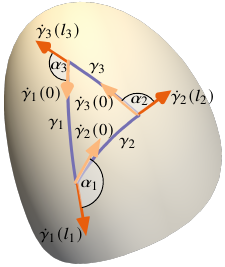}
    \caption{\label{fig:sketch-naming}Illustration of the naming: Geodesic triangle with tangents $\dot{\gamma}_i$ and angles $\alpha_i$ in the corners.}
\vspace{-3em}
\end{wrapfigure}
We consider $q$-sided polygons $\C_q$, with $3\leq q \in \mathbb{N}$. Let $l_i$, $i=1,\ldots,q$ be the length of the sides, $\gamma_i:[0,l_i]\to\M$ a unit speed parametrization of the geodesic that is the $i$-th side,
and $\alpha_i$ the turning angles with respect to the prescribed orientation.

Without loss of generality, we are considering a fiducial point $\gamma_1(t)$, $t\in\rinterval{0}{l_1}$ on the first segment. We split the side $\gamma_1$ into two parts. By abuse of notation, we set $\gamma_1\colon\interval{t}{l_1}\to\M$, and introduce a new segment $\gamma_{q+1}\colon\rinterval{0}{t}\to\M$. Then we set $f_1=0$, and
\begin{align*}
    f_i &\colonequals f_{i-1} + \frac{2\pi \alpha_{i-1}}{\sum_{j=1}^q\alpha_j}=\frac{2\pi }{\sum_{j=1}^q\alpha_j}\sum_{j=1}^{i-1}\alpha_{j}, \quad i=2,\ldots,q+1.
\end{align*}
Obviously, $f_{q+1}=2\pi$. Note that $\sum_{j=1}^q\alpha_j\neq 0$ due to our requirement that the integrated Gaussian curvature $\int_\S K\dA < 2\pi$.
The transported normals $\tconormal_i$, $i=1,\ldots,q+1$, are then defined as
\begin{align}
    \tconormal_i &= \cos(f_i)\conormal_1(t) + \sin(f_i)\tangent_1(t), \label{eq:turned-conormals-geodesic-polygon}
\end{align}
with $\tangent_1=\dot{\gamma}_1$ the tangent vector along the first geodesic segment.
Again, it is easy to see that $\tconormal_1=\tconormal_{q+1}=\conormal_1(t)$.

The definition of the independent components~\eqref{eq:integrated-components-smooth} then reduces to the formulas
\begin{align}\label{eq:integrated-components-polygon}
    \overline{g}_{p,1} &= \sum_{i=1}^{q}l_i\cos(p f_i),&
    \overline{g}_{p,2} &= \sum_{i=1}^{q}l_i\sin(p f_i).
\end{align}
Correspondingly, we find the eigenvalues of the associated irreducible surface MT as
\begin{align}
    \lambda^{\pm}_p\colonequals \pm \sqrt{\overline{g}_{p,1}^2+ \overline{g}_{p,2}^2} \label{eq:eigenvalues-geodesic-polygon}
\end{align}
and eigenvectors as in~\eqref{eq:minkowski-eigenvectors}. Similar as before, we can introduce the normalized eigenvalues $\mu_p \colonequals |\lambda^{\pm}_p|/L$, where $L=\sum_{i=1}^{q}l_i$.

\subsection{Regular polygons}
We consider the special case that all turning angles of the polygon are equal, i.e., $\alpha_i\equiv\alpha$. We immediately get
\begin{align*}
    f_i =\frac{2(i-1)}{q}\pi, \quad i=1,\ldots,q+1.
\end{align*}
This means,
\begin{align*}
    \overline{g}_{p,1}& = \sum_{i=1}^{q}l_i\cos\left(\frac{p}{q}2(i-1)\pi\right), &
    \overline{g}_{p,2}& = \sum_{i=1}^{q}l_i\sin\left(\frac{p}{q}2(i-1)\pi\right).
\end{align*}

In the following theorem we discuss the properties of regular polygons with equal angles and with equal segment lengths. It turns out that these properties lead to extreme eigenvalues of the irreducible surface MT.\@

\begin{theorem}[Eigenvalue for equal angle polygons]\label{thm:eigenvalues-equal-angle-polygon}
    For an equal angle polygon $\C_q$ with $q$ vertices the normalized eigenvalues $\mu_p$ of the irreducible surface MT $\IF{\W_1^p(\C_q)}$ fulfill
    \[
      \mu_p = \left\{\begin{array}{ll}
          1 & \text{ if }\frac{p}{q} \in\mathbb{N} \\
          0 & \text{ if }\frac{p}{q} \notin\mathbb{N}\text{ and }l_i=\frac{L}{q}.
      \end{array}\right.
    \]
\end{theorem}
\begin{proof}
    For $\frac{p}{q} \in\mathbb{N}$, we can easily see $\overline{g}_{q,1}(t)= L$, $\overline{g}_{q,2}(t)= 0$.
    With $\lambda_0^+ = L$, this implies $\lambda^{+}_q/\lambda_0^+ = 1$.

    If $\frac{p}{q}\notin\mathbb{N}$, then
    \begin{align*}
        \sum_{i=2}^{q}\sin\left(\frac{p}{q}2(i-1)\pi\right) & = \frac{\sin\left(\frac{p(q-1)}{q}\pi\right) \sin\left(p\pi\right) }{\sin\left(\frac{p}{q}\pi\right)} = 0,
    \end{align*}
    and
    \begin{align*}
        1+ \sum_{i=2}^{q}\cos\left(\frac{p}{q}2(i-1)\pi\right)
        & =\frac{\sin\left(p\pi\right)\cos\left(\frac{p(q-1)}{q}\pi \right) }{\sin\left(\frac{p}{q}\pi\right)} = 0.
    \end{align*}
    Thus, for equal angles and equal lengths of the sides, $\lambda^{\pm}_p=0$ if $\frac{p}{q}\notin\mathbb{N}$.
\end{proof}

\begin{remark}
    The assertion of \Cref{thm:eigenvalues-equal-angle-polygon} allows to characterize polygonal shapes on surfaces. This is in analogy to the flat case. However, not all surfaces allow for polygons with equal angles and equal segment lengths. Numerical examples are discussed in \Cref{sec:shape-characterization}.
\end{remark}

\subsection{Stability of surface Minkowski tensors}
As motivated by~\cite{SchroderTurk2009Tensorial} in flat space, robustness of the MT against small perturbations is important, see also~\cite{Mickel_JCP_2013,happel2025quantifying}. To demonstrate robustness for the surface MT we consider a small deformation of a geodesic polygon $C_q$ by adding an additional vertex $\xb_{k+1/2}$ inside or outside the polygon between the corners $\xb_k=\gamma_k(0)$ and $\xb_{k+1}=\gamma_k(l_k)$. These three points form a geodesic triangle in itself, with turning angles $\beta_i,\;i=1,2,3$ associated to the points $\xb_k,\xb_{k+1/2},\xb_{k+1}$, respectively. Here, we orient the additional triangle in a way, that is consistent with the orientation of $C_q$.
For the enclosed triangular subdomain $\triangle\subset\M$ we set
\begin{align*}
    2\pi \pm \sum_{i=1}^{3}\beta_i=\int_{\triangle}K\dA \equalscolon \kappa_\epsilon,
\end{align*}
where the $\pm$-sign is negative, if the new vertex is outside the polygon, and positive if the new vertex is inside.
Obviously, $ \kappa_\epsilon=0$ if $K\equiv 0$, and $ \kappa_\epsilon \to 0$ if the enclosed area of the triangle goes to 0.
The new vertex $\xb_{k+1/2}$ also defines an extended polygonal curve, denoted by $\tilde{\C}_{q+1}$. For $|\kappa_\epsilon|<\kappa_0$ small enough, the enclosed total Gaussian curvature $\int_{\tilde{\S}_{q+1}} K\dA < 2\pi$.

\begin{lemma}[Continuity of surface Minkowski tensors]\label{lem:continuity}
    Let $\C_q$ be a geodesic polygon and $\tilde{\C}_{q+1}$ a perturbation as above with $|\kappa_\epsilon|<\kappa_0$. The irreducible surface MT $\IF{\W_1^p(\tilde{C}_{q+1})}$ of the perturbed polygon converge to $\IF{\W_1^p(\C_{q})}$ as the perturbation $\kappa_\epsilon$ goes to zero, i.e., $\bar{g}_{p,j}(\tilde{\C}_{q+1}) \to \bar{g}_{p,j}(\C_{q})$ for $\beta_1\to \mp \pi$, $\beta_2\to 0$, and $\beta_3\to \mp \pi$, where the sign is positive, if the new vertex $\xb_{k+1/2}$ is outside $C_q$, and negative if $\xb_{k+1/2}$ is inside.
\end{lemma}
\begin{proof}
    With $\sum_{j=1}^{q+1}\tilde{\alpha}_j = \sum_{j=1}^{q}{\alpha}_j \pm \kappa_\epsilon$
    we have
    \begin{align*}
        R_\epsilon &\colonequals \frac{\sum_{j=1}^{q}{\alpha}_j}{\sum_{j=1}^{q+1}\tilde{\alpha}_j} \to 1.
    \end{align*}
    The $\tilde{f}_i$ of the extended polygon $\tilde{\C}_{q+1}$ and the $f_i$ of $\C_q$ fulfill the relations
    \begin{align*}
        \tilde{f}_i & =R_\epsilon f_i, \quad (i=1,\ldots, k-1),\\
        \tilde{f}_k & =R_\epsilon \left( f_k + \frac{2\pi}{\sum_{j=1}^{q}{\alpha}_j} (\beta_1\pm \pi)\right),\\
        \tilde{f}_{k+1} & =R_\epsilon \left( f_k + \frac{2\pi}{\sum_{j=1}^{q}{\alpha}_j} (\beta_1\pm \pi+\beta_2)\right),\\
        \tilde{f}_{i} & =R_\epsilon \left( f_{i-1} \pm \kappa_{\epsilon}  \frac{2\pi}{\sum_{j=1}^{q}{\alpha}_j}\right),\quad (i=k+2,\ldots, q+1).
    \end{align*}
Thus, $\tilde{f}_i$ converges to $f_i$ from which the convergence of the $\bar{g}_{p,j}$ readily follows.
\end{proof}

\section{Convergence properties for discretized curves and surfaces}\label{sec:approximation}
In applications neither the curve $\C$ nor the surface $\M$ might be given exactly. We therefore analyze certain approximation schemes, and provide convergence rates for these. We start by approximating the smooth curve $\C$ by a sequence of geodesic polygons and show that the independent components of the irreducible surface MT converge linearly in the maximal segment length. Next we use an approximation with straight lines if the surface is embedded in $\R^3$ and show the same convergence rate as long as we calculate the angles in the tangent planes to the surface. Lastly,  we show that even a piecewise linear approximation of the surface $\M$ within this scheme still has linear convergence rates.

\subsection{Approximation of curves by geodesic polygons}\label{sec:geodesic-approximation}

\begin{wrapfigure}{R}{0.3\textwidth}
\centering
\vspace{-2em}
    \includegraphics[width=\linewidth]{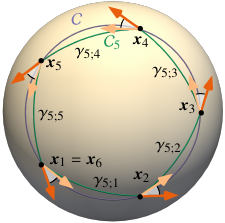}
    \caption{\label{fig:sketch-naming-geodesic-approximation}Illustration of the naming: $\dot \gamma_{p;i}(0)$ in light orange and $\dot \gamma_{p;i}(l_i)$ in dark orange.}
\vspace{-2em}
\end{wrapfigure}
Let $0 = s_1 < \cdots < s_{q+1} = L$ be a discretization of the parameter interval $\interval{0}{L}$ of the closed curve $\C$ and $\xb_i = \gamma(s_i)\in\M$ the associated points. We use a circular numbering, e.g., $\xb_1=\xb_{q+1}$, and denote by $\delta_i=s_{i+1}-s_{i}$ the segment lengths of the curve $\C$ and by $l_i=d_\M(\xb_{i}, \xb_{i+1})$, $i=1,\ldots,q$, the geodesic distances on $\M$ between neighboring points. Let $\C_q$ be the geodesic polygon with vertices $\xb_i$, total length $L_q=\sum_{i=1}^q l_i$, and unit-speed parametrized segments $\gamma_{q;i}\colon\interval{0}{l_i}\to\M$ with $\gamma_{q;i}(0)=\xb_{i}$ and $\gamma_{q;i}(l_i)=\xb_{i+1}$. We assume that $l_i$ is less than the injectivity radius of $\M$ at $\xb_{i}$ for all $i=1,\ldots,q$, so that we can write
\begin{align*}
    \gamma_{q;i}(t)=\exp_{\xb_{i}}\left(\frac{t}{l_i}\log_{\xb_{i}}\xb_{i+1}\right), \quad i=1,\ldots,q;\;t\in[0,l_i],
\end{align*}
with Riemannian $\exp$ and $\log$ associated to the manifold $\M$. For each of the turning angles $\alpha_i$ holds
\begin{align}
    \cos(\alpha_{i}) & = \langle \dot \gamma_{q;{i-1}}(l_{i-1}), \dot \gamma_{q;i}(0)\rangle = - \frac{\langle \log_{\xb_i} \xb_{i-1},\log_{\xb_i} \xb_{i+1} \rangle}{l_{i-1}\;l_{i}} \label{eq:riemannian-log-geodesic-angle}
\end{align}
with the sign of the angles prescribed by the orientation, i.e., a positive angle $\alpha_i$ if the segment bends towards the inside and a negative angle $\alpha_i$ if it bends towards the outside of the shape.

Assuming that the polygon $C_q$ fulfills the requirements introduced in \Cref{sec:geometric-setup}, we construct the irreducible surface MT as in \Cref{sec:geodesic-polygons}.
Without loss of generality, we enumerate our segments so that the first one starts at $\gamma(t=0)$, and we choose the fiducial point to be at $t=0$. Note that, as above, we have $f_1=0$ and do not have the $(q+1)$th side in this case.
Note that the independent components, $\overline{g}_{p,1}$ and $\overline{g}_{p,2}$, are fully determined by the angles, $\alpha_i$, and the local distances, $l_i$. Thus, the quality of the approximation depends on how well the curvature and segment lengths are approximated by the angles and local distances, respectively.

\begin{remark}[Length approximation]\label{rem:length_consistency}
    For the difference of the local lengths holds $\delta_i\geq l_i$, and for $l_i$ small enough
    \begin{align*}
        \delta_i-l_i \leq C \delta_i^3,
    \end{align*}
    where the constant depends on the maximal curvature of $\gamma|_{[s_i,s_{i+1}]}$.
    This also implies
    \begin{align*}
        L-L_q \leq C L \delta^2.
    \end{align*}
\end{remark}

\begin{lemma}[Angle approximation]\label{lem:geodesic_angles}
    Let $\gamma:[0,L]\to M$ be of class $C^3$, and let $0=s_1<s_2<\cdots<s_{q+1}=L$ be a partition with $\delta=\max_{i=1,\ldots,q}\delta_i$ small enough. Let $\gamma_q$ be the closed geodesic polygon with corners $\gamma(s_i)$, $i=1,\ldots,q$, and turning angles $\alpha_i$. Then for all $i=1,\ldots, q+1$, we have
        \begin{align*}
            \alpha_i = \frac{1}{2}k_g(s_i)(\delta_{i}+\delta_{i-1}) + O({\max\{|\delta_i|,|\delta_{i-1}|\}}^2).
        \end{align*}
\end{lemma}

\begin{proof}
    This is the result of a Taylor expansion of the angle at $s_i$ as a function depending on $s_{i-1}$ and $s_{i+1}$.
\end{proof}

\begin{remark}\label{rem:geodesic_angles}
    If $\C$ is only of class $C^2$, the angles still converge to the curvature, but the convergence order will no longer be quadratic.
    If $\C$ is of class $C^4$ and the partition is equidistant, then
        \begin{align*}
            \alpha_i = \frac{1}{2}k_g(s_i)(\delta_{i}+\delta_{i-1}) + O({\max\{|\delta_i|,|\delta_{i-1}|\}}^3)
        \end{align*}
    holds for all $i$.
\end{remark}

\begin{lemma}[Accumulated curvature]\label{lem:geodesic_approx}
    Let $\gamma:[0,L]\to M$ be of class $C^3$, and let $0=s_1<s_2<\cdots<s_{q+1}=L$ be a partition with $\delta=\max_{i=1,\ldots,q}\delta_i$ small enough. Let $\gamma_q$ be the closed geodesic polygon with corners $\gamma(s_i)$, $i=1,\ldots,q$, and turning angles $\alpha_i$. Then there exists a constant $C$ such that for all $i=1,\ldots, q+1$, we have
    \begin{align*}
        \left|\int_{0}^{s_i}k_g(\tau)\dtau - \sum_{j=1}^{i-1}\alpha_j\right| \leq C\delta \left(1+\frac{1}{2}\left(k_g(0)+k_g(s_{i})\right) + L\right).
    \end{align*}
    If $\gamma$ is of class $C^4$ and the partition is equidistant, we obtain for the total curvature
    \begin{align*}
        \left|\int_{\gamma}k_g(\tau)\dtau - \sum_{j=1}^{q}\alpha_j\right| \leq C\delta^2.
    \end{align*}
\end{lemma}
\begin{proof}
    The result is stated in~\cite[Thm 3.4]{LopezEtAl2010Total} for $\C$ only in $C^2$ for some $\epsilon$ instead of $C\delta$. As the proof is based on the approximation of integrals by a box rule and the local approximation of the angles (see \Cref{lem:geodesic_angles}), the linear order of convergence for curves in $C^3$ can readily be concluded.

    Using an equidistant partition, i.e., $\delta_i=\delta$ for all $i$, and the improved local approximation (see \Cref{rem:geodesic_angles}) for curves in $C^4$ we see that by a midpoint quadrature rule
    \begin{align*}
        \frac{1}{2}k_g(s_i)(\delta_{i}+\delta_{i-1}) =
        \frac{1}{2}\int_{s_{i-1}}^{s_{i+1}}k_g(\tau)\dtau + \frac{1}{6}k_g''(\xi) \delta^3.
    \end{align*}
    Thus, we have
    \begin{align*}
        \alpha_i = \frac{1}{2}\int_{s_{i-1}}^{s_{i+1}}k_g(\tau)\dtau + O(\delta^3).
    \end{align*}
    Summing up, we get for the total curvature (as $q\delta=L$)
    \begin{align*}
        \sum_{i=0}^q \alpha_i &= \int_{\gamma}k_g(\tau)\dtau + O(\delta^2).\qedhere
    \end{align*}
\end{proof}

A direct consequence of this Lemma is that the sum of the angles of the approximating polygon can be arbitrarily close to the total geodesic curvature $\int_\C k_g\ds = 2\pi - \int_\S K \dA$ of $\C$. Thus, the irreducible surface MT are defined for $\C_q$ if $\delta<\delta_0$ small enough depending on $\int_\C k_g\ds$. However, the constants in the approximation error estimate will still depend on the inverse of the total curvature as stated in the following Corollary. In general, approximation results will fail to hold for the limit case.

\begin{corollary}\label{cor:f-fh}
     Let $s\in\rinterval{s_{i}}{s_{i+1}}$ and $\delta=\max_{i=1,\ldots,q}\delta_i$ be small enough. Then we have for $i\geq 1$
    \begin{align*}
        |f(s,0) - f_i| & \leq C\delta,
    \end{align*}
    where the constant now depends on the inverse of the total curvature of $\gamma$, the curvature at $s_{i+1}$, and, for non-convex curves, on the maximum curvature of $\gamma$ at the nodes $s_j$, $j\leq i$.
\end{corollary}
\begin{proof}
    It is enough to show this for $s=s_i$, as the difference $f(s,0)-f(s_i,0)$ is obviously of the correct order.
    \begin{align*}
        |f(s_i,0) - f_i| & = \frac{2\pi}{|\int_{\gamma}k_g(\tau) \dtau|}\left|\int_{0}^{s_i}k_g(\tau)\dtau - \sum_{j=1}^{i-1}\alpha_j +\left(1- \frac{\int_{\gamma}k_g(\tau) \dtau}{\sum_{j=1}^q\alpha_j}\right)\sum_{j=1}^{i-1}\alpha_j \right|\\
        &\leq \frac{2\pi}{|\int_{\gamma}k_g(\tau) \dtau|}\left|\int_{0}^{s_i}k_g(\tau)\dtau - \sum_{j=1}^{i-1}\alpha_j\right|\\
        &\quad +\frac{2\pi}{|\int_{\gamma}k_g(\tau) \dtau|}\left|\sum_{j=1}^q\alpha_j- \int_{\gamma}k_g(\tau) \dtau \right|\left|\frac{\sum_{j=1}^{i-1}\alpha_j}{\sum_{j=1}^{q}\alpha_j}\right|.
    \end{align*}
For convex curves, the assertion follows directly from \Cref{lem:geodesic_approx} as the last quotient is bounded by one. For non-convex curves, we estimate
\begin{align*}
    \sum_{j=1}^{q}\alpha_j&\geq \int_{\gamma}k_g(\tau)\dtau- C\delta > 0\\
    \left|\sum_{j=1}^{i-1}\alpha_j\right|&\leq \sum_{j=1}^{i-1}|\alpha_j|\leq L\max_{j=1,\ldots,i}|k_g(s_j)| + C\delta L. \qedhere
\end{align*}
\end{proof}

\begin{corollary}\label{cor:g-gh}
    Let $\gamma:[0,L]\to M$ be of class $C^3$, and let $0=s_1<s_2<\cdots<s_{q+1}=L$ be a partition with $\delta=\max_{i=1,\ldots,q}\delta_i$ small enough. Let $\gamma_q$ be the closed geodesic polygon with corners $\gamma(s_i)$, $i=1,\ldots,q$.
    Then
    \begin{align*}
        \left|\overline{g}_{p,j;\gamma}(0)-\overline{g}_{p,j;\gamma_q}\right| & \leq C\delta,\qquad\text{ for }j=1,2.
    \end{align*}
\end{corollary}
\begin{proof}
    We have
    \begin{align*}
        \left|\overline{g}_{p,1;\gamma}(0)- \overline{g}_{p,1;\gamma_q}\right|
        & = \left|\int_{0}^{L}\cos(pf(s,0))\;ds -\sum_{i=1}^{q}l_i \cos(p f_i)\right|\\
        & \leq  \left|\int_{0}^{L}\cos(pf(s,0))\;ds-\sum_{i=1}^{q}\delta_i\cos(p f(s_i,0))\right|\\
        & + p\sum_{i=1}^{q}\delta_i |f(s_i,0)-f_i|+  \sum_{i=1}^{q}\left|\delta_i-l_i\right|.
    \end{align*}
    Now, the first term is again approximation of an integral by a quadrature rule. The second term can be estimated by $C\delta$ by the previous \Cref{cor:f-fh}. The last term compares the distances, and can also be estimated by $L-L_q \leq C\delta^2$.
\end{proof}

\subsection{Approximation of curves by straight line polygons}\label{sec:straight-line-approximation}
\begin{wrapfigure}{R}{0.3\textwidth}
\vspace{-2em}
\centering
    \includegraphics[width=\linewidth]{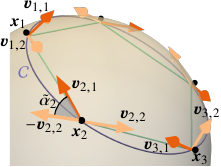}
    \caption{\label{fig:sketch-naming-line-approximation}Illustration of the naming: visualization of $\xb_i-\xb_{i+1}$ in green.}
\vspace{-2em}
\end{wrapfigure}
In the following, we assume that $\M$ is isometrically embedded in $\R^3$ and provides an outer surface normal field $\nb$.

While geodesic polygons allow for a convergent approximation of a smooth curve, the Riemannian $\exp$ and $\log$ are not always readily available, and can be expensive to compute. We therefore consider an even simpler approximation of the curve, by a chain of straight line segments in the embedding space $\R^3$.
In order to construct and evaluate such an approximation, we start with a geodesic polygon $\C_q$ that approximates the smooth curve $\C$ as in \Cref{sec:geodesic-approximation} and construct a corresponding piecewise linear approximation.
As the independent components consist of sums of angles weighted by local distances, almost all consistency errors, i.e., the differences of local distances and of local angles, have to be better than linear in order to obtain convergence. For linear convergence, the consistency errors have to be quadratic. This is usually the case for the local distances, as long as the curvature of all curves involved stays bounded (cf. \Cref{rem:length_consistency}). In order to approximate the angles, which in turn approximate the local curvature (cf. \Cref{lem:angle-approximation}), some information about the underlying surface is needed. This can be provided by the surface normals $\nb$ at the vertices of the polygon, or by approximations thereof.

One possible approximation is the following:
Let $\nb_i$ denote the unit outer normal to $\M$ at the vertices $\xb_i$, $i=1,\ldots,q$. We introduce the tangential vectors
\begin{gather}\label{eq:projected-difference}
\begin{aligned}
    \vb_{i,1} & \colonequals P_{\nb_i}(\xb_{i-1}-\xb_i) = (\xb_{i-1}-\xb_i) - \langle \xb_{i-1}-\xb_i,\nb_i\rangle \nb_i, \\
    \vb_{i,2} & \colonequals P_{\nb_i}(\xb_{i+1}-\xb_i) = (\xb_{i+1}-\xb_i) - \langle \xb_{i+1}-\xb_i, \nb_i\rangle \nb_i
\end{aligned}
\end{gather}
as projections of the vector connecting two neighboring points in the embedding space $\R^3$. We assume further that the distance of the $\xb_i$ is small enough, such that $\langle \xb_{i-1}-\xb_i,\nb_i\rangle^2 \leq \frac{1}{2}\norm{\xb_{i-1}-\xb_i}^2$. Then $\norm{P_{\nb_i}(\xb_{i-1}-\xb_i)}>0$,
and we can set $\tilde{l}_i \colonequals \norm{\xb_i - \xb_{i-1}}$ as approximations of the local distances, and
\begin{align}
    \cos(\tilde{\alpha}_i) &\colonequals - \frac{\langle \vb_{i,1}, \vb_{i,2}\rangle}{\norm{\vb_{i,1}}\;\norm{\vb_{i,2}}}
\end{align}
as an approximation of the angles in the geodesic polygon, with sign as before. For the following construction, we assume that $\sum_{j=1}^q\tilde{\alpha}_j>0$ similar to the requirements for geodesic polygons. As we will show in the following that the $\tilde{\alpha}_i$ approximate the angles $\alpha_i$ of the geodesic polygon well enough, this can again be guaranteed by choosing $\delta$ small enough.
This allows to introduce approximations of the $f_i$ as $\tilde{f}_1=0$ and
\begin{align}
    \tilde{f}_i &\colonequals \frac{2\pi}{\sum_{j=1}^q\tilde{\alpha}_j}\sum_{j=1}^{i-1}\tilde{\alpha}_j, \quad i=2,\ldots, q
\end{align}
and corresponding approximations of the independent components of the irreducible surface MT $\IF{\W_1^p}$,
\begin{align}
    \tilde{g}_{p,1}(t) &\colonequals \sum_{i=1}^{q}\tilde{l}_i\cos(p \tilde{f}_i),&
    \tilde{g}_{p,2}(t) &\colonequals \sum_{i=1}^{q}\tilde{l}_i\sin(p \tilde{f}_i).
\end{align}

As we see, the angle approximation is usually not given directly, but as an approximation of the cosine via a vector approximation.
The following lemma allows us to estimate the angle consistency error in terms of the vector perturbation.

\begin{lemma}\label{lem:angle_vector_pertubation}
Let $\alpha$ be the angle between two unit vectors $V,W\in T_{\xb} \M \subset \R^3$, and let $V_{\delta}\in \R^3$ be a unit vector with $\norm{V-V_{\delta}}\leq \delta$ for some small $\delta>0$. Let $\alpha_{\delta}$ be the angle between $V_{\delta}$ and $W$. Then
    \begin{align*}
        |\alpha-\alpha_{\delta}|\leq
        \delta+ C\delta^3.
    \end{align*}
\end{lemma}

\begin{proof}
    We have by triangle inequality of the distance on the sphere
    \[
    \alpha-\alpha_\delta \leq \arccos(\langle V,V_{\delta}\rangle) = \arccos\left(1-2\norm{V-V_{\delta}}^2\right)\leq \delta + C\delta^3. \qedhere
    \]
\end{proof}

\begin{lemma}[Approximation of the angles]\label{lem:angle-approximation}
    \begin{align}
        |\alpha_i- \tilde{\alpha}_i| &\leq C\;(l_i^2+l_{i+1}^2).
    \end{align}
\end{lemma}
\begin{proof}
    We have for points $\xb,\yb \in \M$ and normal vector $\nb_{\xb}$ in $\xb$
    \begin{align*}
        \Norm{\frac{\log_{\xb}\yb}{d_{\M}(\xb,\yb)} - \frac{P_{\nb_{\xb}}(\yb-\xb)}{\norm{P_{\nb_{\xb}}(\yb-\xb)}}} &\leq C d_{\M}^2(\xb,\yb).
    \end{align*}
    This can be checked by Taylor expansion of the vector field $\vb:[0,1]\to T_{\xb}\M$
    \begin{align*}
        \vb(t) = P_{\nb_{\xb}}(\underbrace{\exp_{\xb}(t \log_{\xb} \yb)}_{\gamma_{\xb,\yb}(t)} - \xb),\quad t\in[0,1].
    \end{align*}
    The estimate follows from \Cref{lem:angle_vector_pertubation}.
\end{proof}

Following the argumentation in \Cref{cor:f-fh} and \Cref{cor:g-gh} we obtain linear convergence of the irreducible components of this scheme.

\subsection{Approximation of the surface}
If we have only approximate knowledge about the underlying surface $\M$, e.g., only a triangulation or a discrete levelset representation is available, the surface normals $\nb$ are also known only approximately.

We consider a perturbation of the straight line polygonal approximation described in \Cref{sec:straight-line-approximation}, by replacing the $\xb_i\in \M$ with $\xb_{h;i}\in \R^3$, and the normal vectors $\nb_{i}$ by $\nb_{h,i}$, for $i=1,\ldots,q$.
The parameter $h$ is associated to the approximation quality in the sense that $\norm{\xb_i-\xb_{h,i}}\leq C h^2$, and $\norm{\nb_{i}-\nb_{h,i}}\leq C h$, for all $i=1,\ldots,q$.
Note that this corresponds to the quality of a piecewise linear approximation of $\M$. We will further assume that $\delta = O(h)$.

In the following, we omit the index $i$ for better readability.
For a vector $\vb\in\R^3$ we have the estimate
\begin{align*}
    \norm{P_{\nb}\vb - P_{\nb_{h}}\vb}
	& \leq 2 \norm{\vb}\; \norm{\nb_{h} - \nb}.
\end{align*}
As long as $\norm{P_{\nb_{h}}\vb}\geq \frac{1}{2}\norm{\vb}$, we also have
\begin{align*}
	\Norm{\frac{P_{\nb}\vb}{\norm{P_{\nb}\vb}} - \frac{P_{\nb_{h}\vb}}{\norm{P_{\nb_{h}}\vb}}}
	& \leq 8 \norm{\nb_{h} - \nb}.
\end{align*}

For a quadratic approximation order of the normals, linear convergence of the scheme would directly follow from \Cref{lem:angle_vector_pertubation}. However, even a linear approximation of the normals is enough to ensure quadratic convergence of the angles, and hence linear convergence of the surface MT.\@

\begin{lemma}\label{lem:PhvPhw-vw}
    Let $\vb,\wb\in T_{\xb}\M$ with $\norm{\vb}=\norm{\wb}=1$, normal vector $\nb=\nb_{\xb}$, and a point $\xb_h$ associated to $\xb$ with normal vector $\nb_h=\nb_{h,\xb_h}$. Then
    \begin{align*}
        \left|\frac{\langle \PPh\vb,\PPh\wb\rangle}{\norm{\PPh\vb}\norm{\PPh\wb}} - \langle\vb,\wb \rangle \right|\leq C h^4 + Ch^2 \norm{\vb-\wb}^2.
    \end{align*}
\end{lemma}
\begin{proof}
    Note that $\norm{\PPh\vb} = \sqrt{1-\langle \vb,\nb_h\rangle^2} =1+O(h)$, and analogously for $\PPh\wb$.
    We rewrite
    \begin{align*}
        \langle \PPh\vb,\PPh\wb\rangle- \norm{\PPh\vb}\norm{\PPh\wb}\langle\vb,\wb \rangle
            &= \langle\vb,\wb \rangle\left(1-\norm{\PPh\vb}\norm{\PPh\wb} - \langle \vb,\nb_h\rangle\langle \wb,\nb_h\rangle \right) \\
            &\phantom{=} - \frac{1}{2}\norm{\vb-\wb}^2\langle \vb,\nb_h-\nb\rangle\langle \wb,\nb_h-\nb\rangle.
    \end{align*}
    Using a Taylor expansion of $x\mapsto \sqrt{1-x^2}$, we can estimate
    \begin{multline*}
        \left|
            1-\norm{\PPh\vb}\norm{\PPh\wb} - \langle \vb,\nb_h\rangle\langle \wb,\nb_h\rangle
        \right| \\
        \begin{aligned}
        &\leq \left|
            \frac{1}{2} (\langle \vb,\nb_h\rangle^2+\langle \wb,\nb_h\rangle^2) - \langle \vb,\nb_h\rangle\langle \wb,\nb_h\rangle
        \right|+Ch^4\\
        &\leq Ch^2 \norm{\vb-\wb}^2 +  C h^4,
        \end{aligned}
    \end{multline*}
    from which the assertion follows.
\end{proof}

\begin{lemma}\label{lem:PhvPhw-PvPw}
    Let $\vb,\wb\in \R^3$ be two vectors at a point $\xb\in \M$ with associated normal vector $\nb=\nb_{\xb}$, such that $|\langle \vb, \nb\rangle |\leq ch\norm{\vb}$ and $|\langle \wb, \nb\rangle |\leq ch\norm{\wb}$, and let $\xb_h$ be a point associated to $\xb$ with normal vector $\nb_h=\nb_{h,\xb_h}$. Then
    \begin{equation*}
        \left|
          \frac{\langle \PPh\vb,\PPh\wb\rangle}{\norm{\PPh\vb}\norm{\PPh\wb}}
          -\frac{\langle \PP\vb,\PP\wb\rangle}{\norm{\PP\vb}\norm{\PP\wb}}
        \right|
        \leq C h^4 + Ch^2\Norm{\frac{\PP\vb}{\norm{\PP\vb}} -\frac{\PP\wb}{\norm{\PP\wb}}}.
    \end{equation*}
\end{lemma}
\begin{proof}
    We set $\tilde{\vb}\colonequals \frac{\PP\vb}{\norm{\PP\vb}}$, and $\tilde{\wb}\colonequals \frac{\PP\wb}{\norm{\PP\wb}}$. By \Cref{lem:PhvPhw-vw}, we have
    \begin{align*}
        \left|\frac{\langle \PPh\tilde{\vb},\PPh\tilde{\wb}\rangle}{\norm{\PPh\tilde{\vb}}\norm{\PPh\tilde{\wb}}} - \frac{\langle \PP\vb,\PP\wb\rangle}{\norm{\PP\vb}\norm{\PP\wb}} \right|\leq C h^4 + Ch^2 \norm{\tilde{\vb}-\tilde{\wb}}^2.
    \end{align*}
    We set further $\hat{\vb}\colonequals \frac{\PPh\vb}{\norm{\PPh\vb}}$, $\hat{\wb}\colonequals \frac{\PPh\wb}{\norm{\PPh\wb}}$, $\hat{\vb}_h \colonequals \frac{\PPh\PP\vb}{\norm{\PPh\PP\vb}}$, and $\hat{\wb}_h\colonequals \frac{\PPh\PP\wb}{\norm{\PPh\PP\wb}}$.
    Then
    \begin{align*}
        \frac{\langle \PPh\tilde{\vb},\PPh\tilde{\wb}\rangle}{\norm{\PPh\tilde{\vb}}\norm{\PPh\tilde{\wb}}}
        & = \langle\hat{\vb}_h,\hat{\wb}_h\rangle,
    \end{align*}
    and from $\norm{\PPh \vb-\PPh\PP\vb}=|\langle \vb,\nb\rangle|\norm{\PPh\nb}\leq Ch^2\norm{\vb} \leq Ch^2\norm{\PPh \vb}$ follows ${\norm{\hat{\vb}-\hat{\vb}_h}\leq Ch^2}$, and with a similar argument also $\norm{\hat{\wb}-\hat{\wb}_h}\leq Ch^2$. Further, we have
    \begin{align*}
        |\langle \hat{\vb}-\hat{\vb}_h, \hat{\wb}\rangle|& \leq \frac{1}{2}\norm{\hat{\vb}-\hat{\vb}_h}^2 + \frac{1}{4}\norm{\hat{\vb}-\hat{\vb}_h}
        \norm{\hat{\vb}-\hat{\wb}}\\
        & \leq Ch^4 + Ch^2
        \norm{\hat{\vb}-\hat{\wb}}.
    \end{align*}
    Thus, we get
    \begin{align*}
      \left|
          \langle\hat{\vb},\hat{\wb}\rangle - \langle\hat{\vb}_h,\hat{\wb}_h\rangle
       \right|
       &\leq Ch^4 + Ch^2 \norm{\hat{\vb}-\hat{\wb}} + Ch^2 \norm{\hat{\vb}_h-\hat{\wb}_h}.
    \end{align*}
    As we can estimate
    \begin{align*}
      \norm{\hat{\vb}-\hat{\wb}}\leq Ch^2 + \norm{\hat{\vb}_h-\hat{\wb}_h} \leq Ch^2 + (1+Ch) \norm{\tilde{\vb}-\tilde{\wb}},
    \end{align*}
    the assertion follows.
\end{proof}

\begin{corollary}\label{cor:linear-convergence-line-segment-approx}
    Let $\vb,\wb\in \R^3$ be two vectors at a point $\xb\in\M$ with associated normal vector $\nb=\nb_{\xb}$, such that $|\langle \vb, \nb\rangle |\leq ch\norm{\vb}$ and $|\langle \wb, \nb\rangle |\leq ch\norm{\wb}$, and let $\vb_h,\wb_h\in \R^3$ be two vectors at a point $\xb_h$ with associated normal vector $\nb_h=\nb_{h,\xb_h}$, such that $\norm{\vb-\vb_h} + \norm{\wb-\wb_h}\leq Ch^2$. Then
    \begin{equation*}
        \left|
          \frac{\langle \PPh\vb_h, \PPh\wb_h\rangle}{\norm{\PPh\vb_h}\norm{\PPh\wb_h}}
          -\frac{\langle \PP\vb, \PP\wb\rangle}{\norm{\PP\vb}\norm{\PP\wb}}
        \right|
        \leq C h^4 + Ch^2\Norm{\frac{\PP\vb}{\norm{\PP\vb}} -\frac{\PP\wb}{\norm{\PP\wb}}}.
    \end{equation*}
\end{corollary}

Setting $\vb\colonequals \xb_{i-1} - \xb_{i}$, $\wb\colonequals \xb_{i} - \xb_{i+1}$,  $\vb_h\colonequals \xb_{h,i-1} - \xb_{h,i}$, and $\wb_h\colonequals \xb_{h,i} - \xb_{h,i+1}$, all associated to the base points $\xb_i$ and $\xb_{h,i}$, respectively, this implies for the angles of the approximate scheme indeed a quadratic error in $h$. Thus, linear convergence of the irreducible components of the surface MT follows.

\section{Numerical experiments}\label{sec:numerical-experiments}
We study the irreducible surface MT on various surfaces to confirm the ability to characterize shapes by the eigenvalues of the irreducible surface MT, to study the convergence properties of the various approximations, and to demonstrate the applicability in a real-world example in biology. We further want to highlight that the irreducible surface MT are easy to implement. An implementation in a Julia package, which is used for all numerical experiments, is provided in \texttt{SurfaceMinkowski.jl}~\cite{SurfaceMinkowski-jl}.

We consider three types of surfaces:
\begin{enumerate}
    \item The sphere with analytically known geodesics along the great circle arcs and explicit formulas for Riemannian distance, $\exp$, and $\log$;
    \item Parametrized surfaces with a single chart, ellipsoid and torus, where the geodesics are given implicitly as solutions of ordinary differential equations in intrinsic coordinates or from an embedding of the surface in $\R^3$;
    \item Surfaces given as a triangulation with curves implicitly represented as the levelset of a discrete function on these surfaces.
\end{enumerate}

\subsection{Validation of shape characterization}\label{sec:shape-characterization}
The first verification of the method is on the sphere. We study regular polygons with vertices on a geodesic circle. \Cref{thm:eigenvalues-equal-angle-polygon} says that for a geodesic polygon with $p$ vertices and all angles and edge lengths equal, the normalized eigenvalues of the irreducible surface MT $\mu_p=1$ as well as all $\mu_m=1$ for $m$ a multiple of $p$, but the other eigenvalues should be 0. This can be observed in the numerical experiment visualized in \Cref{fig:sketch-RegularSphere}. We measured the normalized eigenvalues $\mu_p$ for all $p\in\{1,\ldots,6\}$ for regular polygons with vertices from $3,\ldots,6$. For the geodesic triangle, both $\mu_3$ and $\mu_6$ are measured as $1.0$, and for the other polygons, the normalized eigenvalue corresponding to the number of vertices is measured as $1.0$. All other eigenvalues in the tested range are, as expected, measured as 0. These results are computed by the formula~\eqref{eq:eigenvalues-geodesic-polygon} for the normalized Minkowski eigenvalues of a geodesic polygon. In the formula only the angles between tangents at the corner vertices are involved. Those could be measured by the inner product of the Riemannian log between neighboring corner points, compare~\eqref{eq:riemannian-log-geodesic-angle}. For the sphere, the $\log_{\xb}\yb \simeq P_{\xb}(\yb-\xb)$ is related to the projected difference between the point coordinates, see~\eqref{eq:projected-difference}, and can thus easily be computed. A geodesic polygon on the sphere is also simply constructed by an equi-distribution of points on a geodesic circle.

\begin{figure}
\centering
    \begin{subfigure}[t]{0.22\textwidth}
        \includegraphics[width=\textwidth]{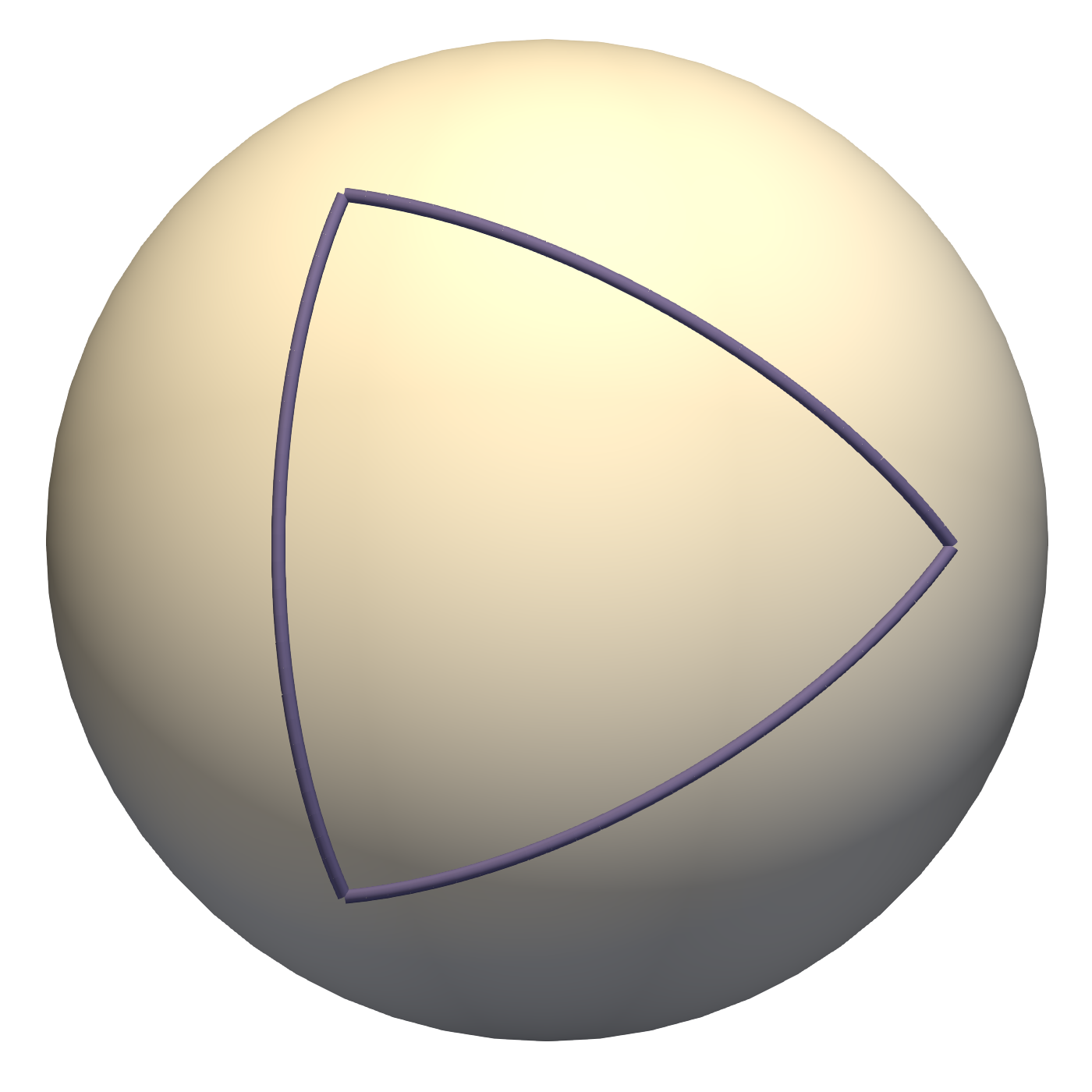}
        \subcaption{$\mu_3=\mu_6=1.0$}
    \end{subfigure}%
    \begin{subfigure}[t]{0.22\textwidth}
        \includegraphics[width=\textwidth]{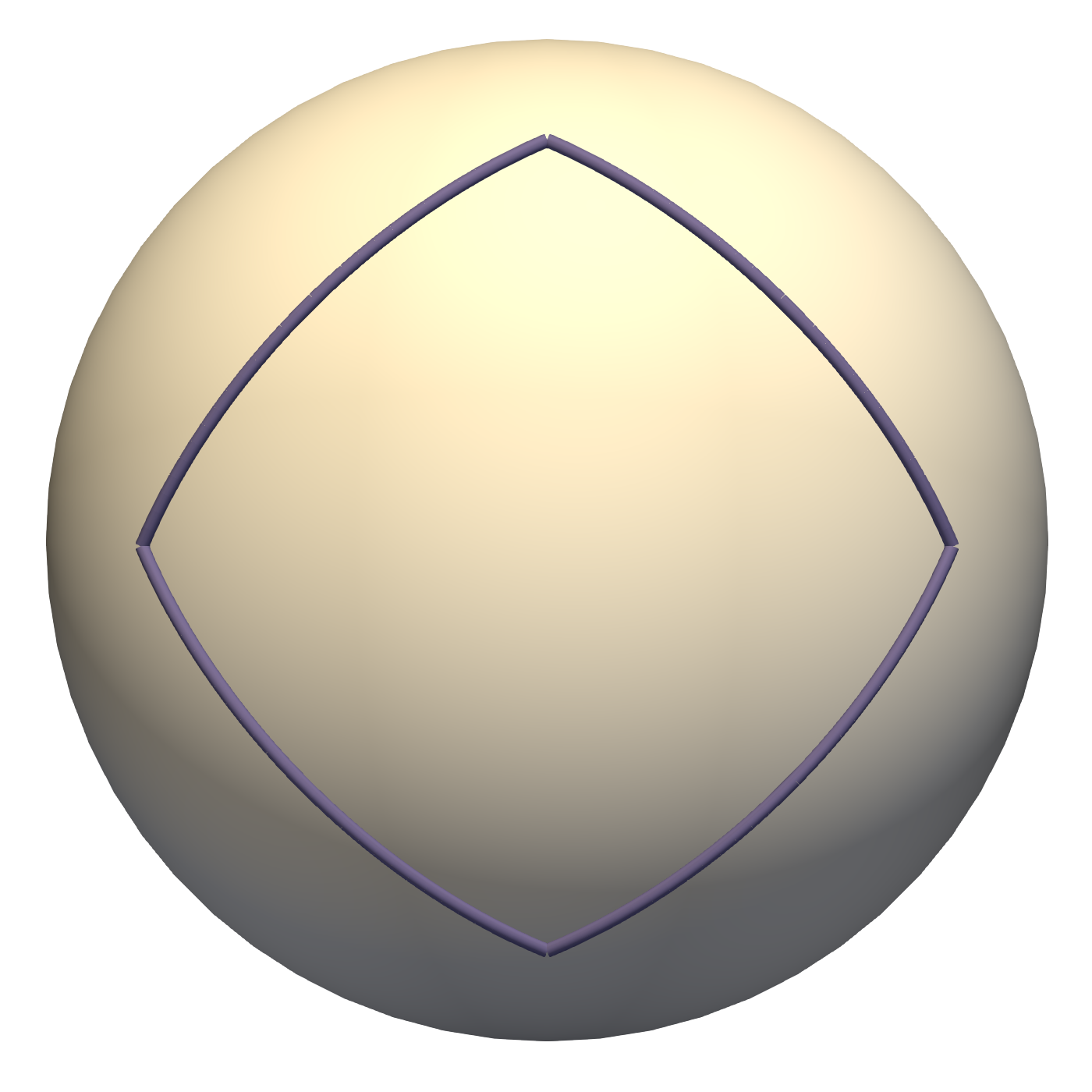}
        \subcaption{$\mu_4=1.0$}
    \end{subfigure}%
    \begin{subfigure}[t]{0.22\textwidth}
        \includegraphics[width=\textwidth]{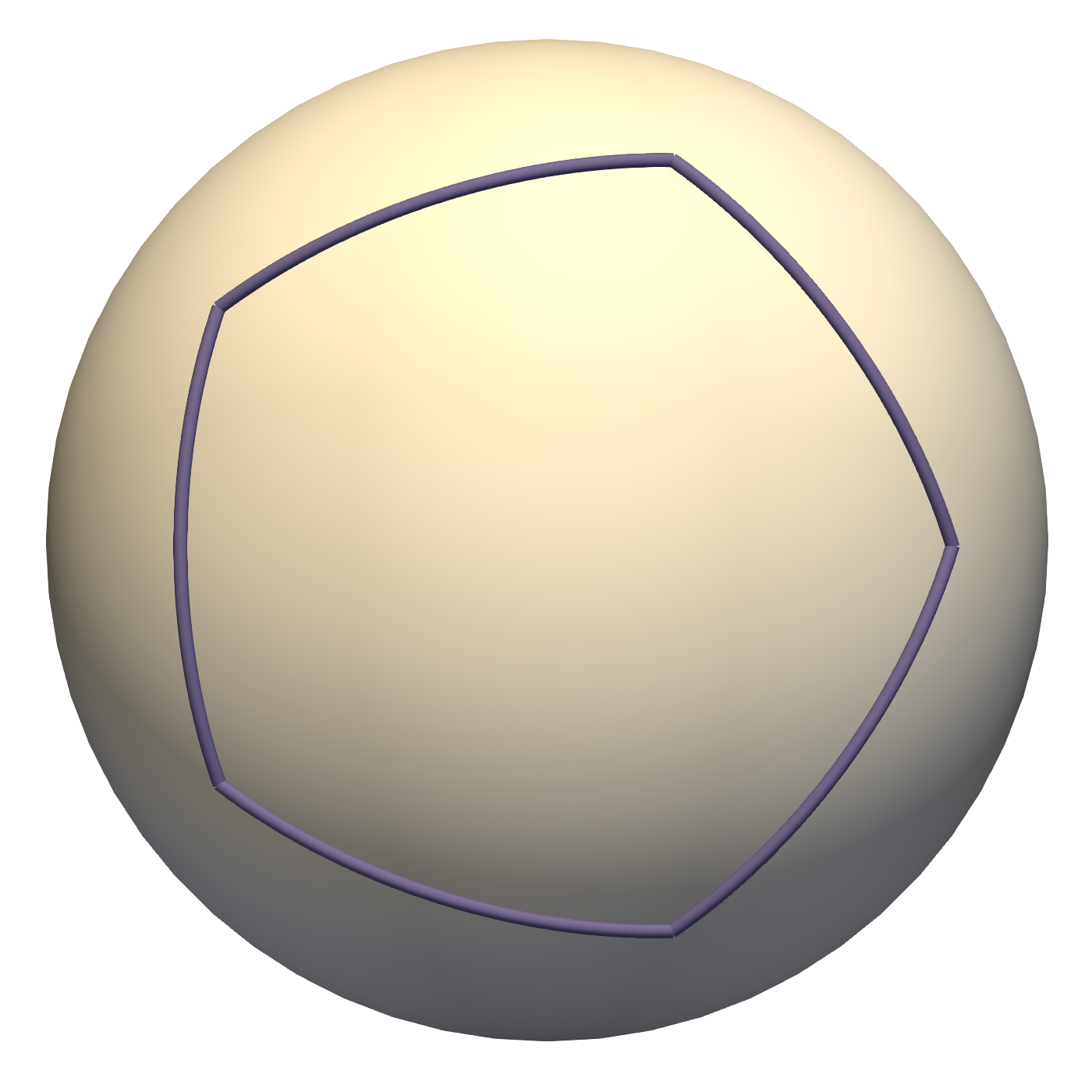}
        \subcaption{$\mu_5=1.0$}
    \end{subfigure}%
    \begin{subfigure}[t]{0.22\textwidth}
        \includegraphics[width=\textwidth]{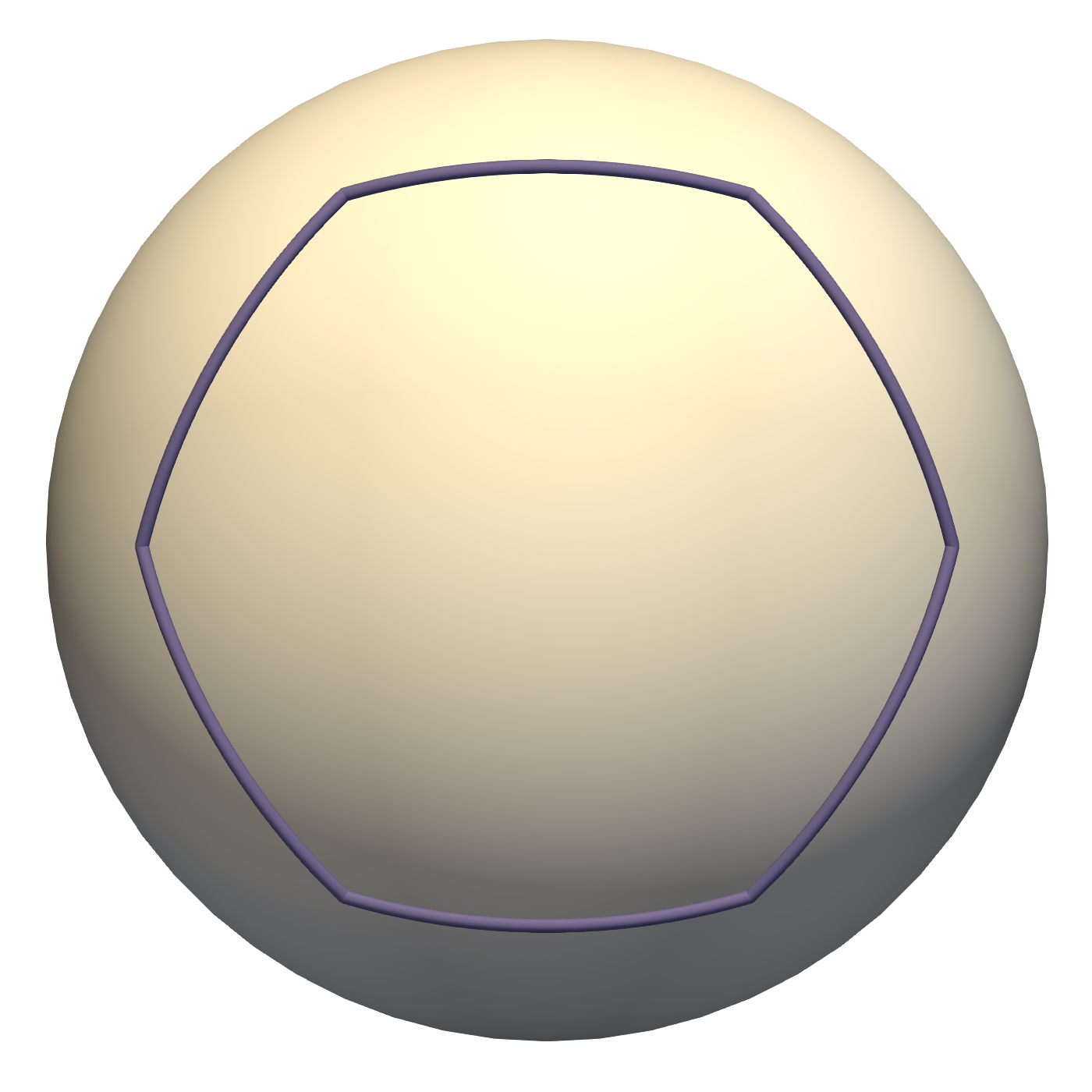}
        \subcaption{$\mu_6=1.0$}
    \end{subfigure}%
    \caption{\label{fig:sketch-RegularSphere} Regular polygons on a sphere. All $\mu_p$ with $p \leq 6$ that are not indicated are equal to zero.}
\end{figure}

The normalized eigenvalues are constant along the curve, i.e., independent of the fiducial point, as described in \Cref{lem:point-independence-eigenvalue}. Another property, shown in \Cref{lem:point-independence-eigenvector}, is that the eigenvectors of the irreducible surface MT are constant in the sense that the turned covariant derivative of the eigenvectors vanishes along the curve. This means that the directions can be transported along the curve using the turned transport $\tPt^\gamma$. This is in contrast to the classical parallel transport $\Pt^\gamma$, which results in a rotated direction after a full circle transport for non-flat surfaces.
In \Cref{fig:ParallelTransport} we illustrate the difference in the transport of the curve co-normal $\conormal$ with the standard parallel transport in (a) and turned transport in (b) along a geodesic polygon on the sphere. It is highlighted that the two transports are different and that the turned transport defines a closed curve, i.e., the transported vector end at the starting vector after a full circle transport, see (b).

\begin{figure}
    \centering
    \begin{subfigure}[t]{0.22\textwidth}
        \includegraphics[width=\textwidth]{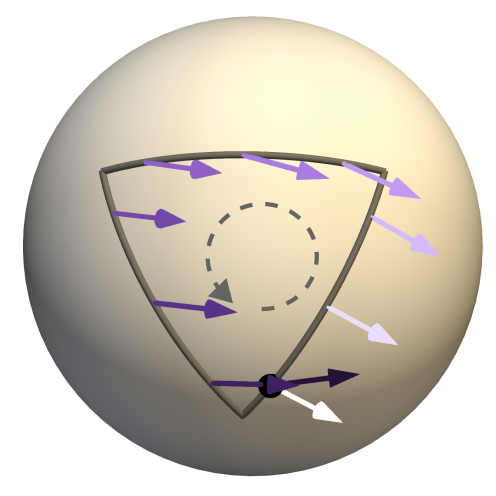}
        \subcaption{\label{fig:ParallelTransport-a}}
    \end{subfigure}%
    \begin{subfigure}[t]{0.22\textwidth}
        \includegraphics[width=\textwidth]{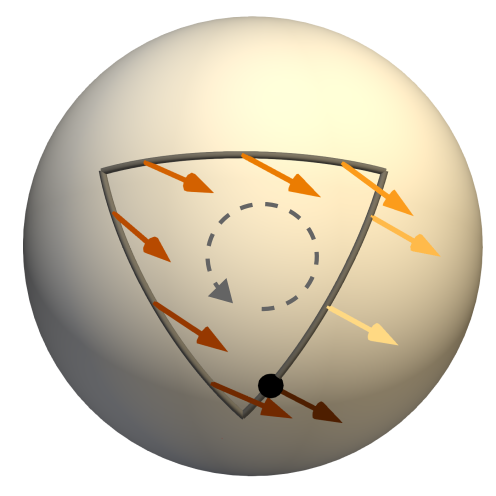}
        \subcaption{\label{fig:ParallelTransport-b}}
    \end{subfigure}%
    \begin{subfigure}[t]{0.22\textwidth}
        \includegraphics[width=\textwidth]{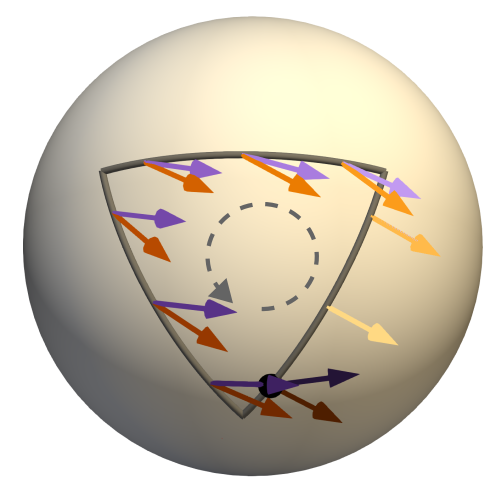}
        \subcaption{\label{fig:ParallelTransport-c}}
    \end{subfigure}%
    \caption{Parallel transport (a) and turned transport with additional rotation (b) of the co-normal vector $\conormal$ starting from the point marked with a black dot in positive direction. Colors indicate the progress in the transport. In (c) both transports are shown at the same time for better comparison.}\label{fig:ParallelTransport}
\end{figure}

\begin{figure}
 \centering
    \begin{subfigure}[t]{0.5\textwidth}
    \setcounter{subsubfigure}{0}%
    \begin{subsubfigure}{0.49\textwidth}
        \includegraphics[height=3.1cm]{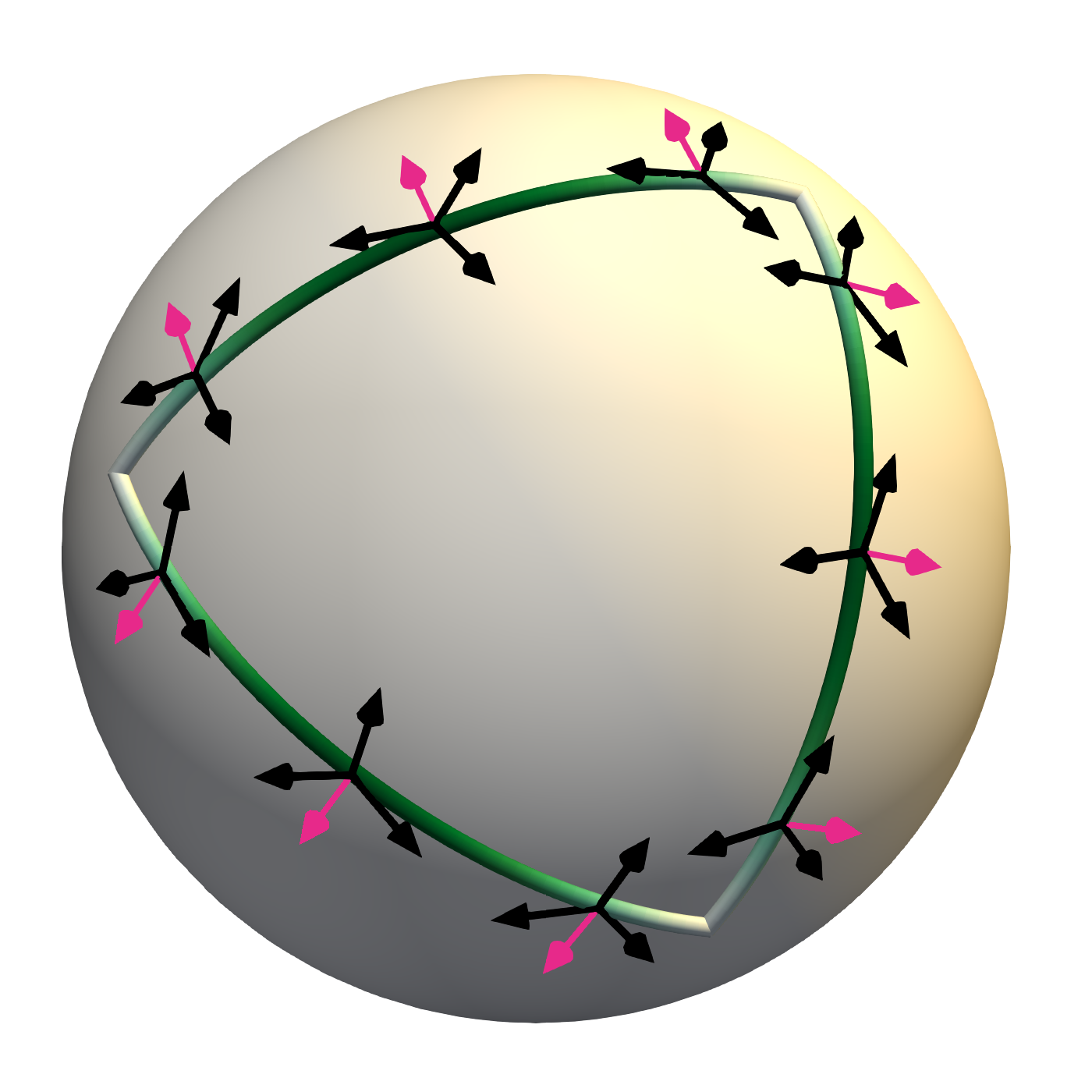}
        \subcaption{$\Pt^\gamma$, $0.23\!<\!\mu_3\!<\!0.34$}
    \end{subsubfigure}%
    \begin{subsubfigure}{0.48\textwidth}
        \includegraphics[height=3.1cm]{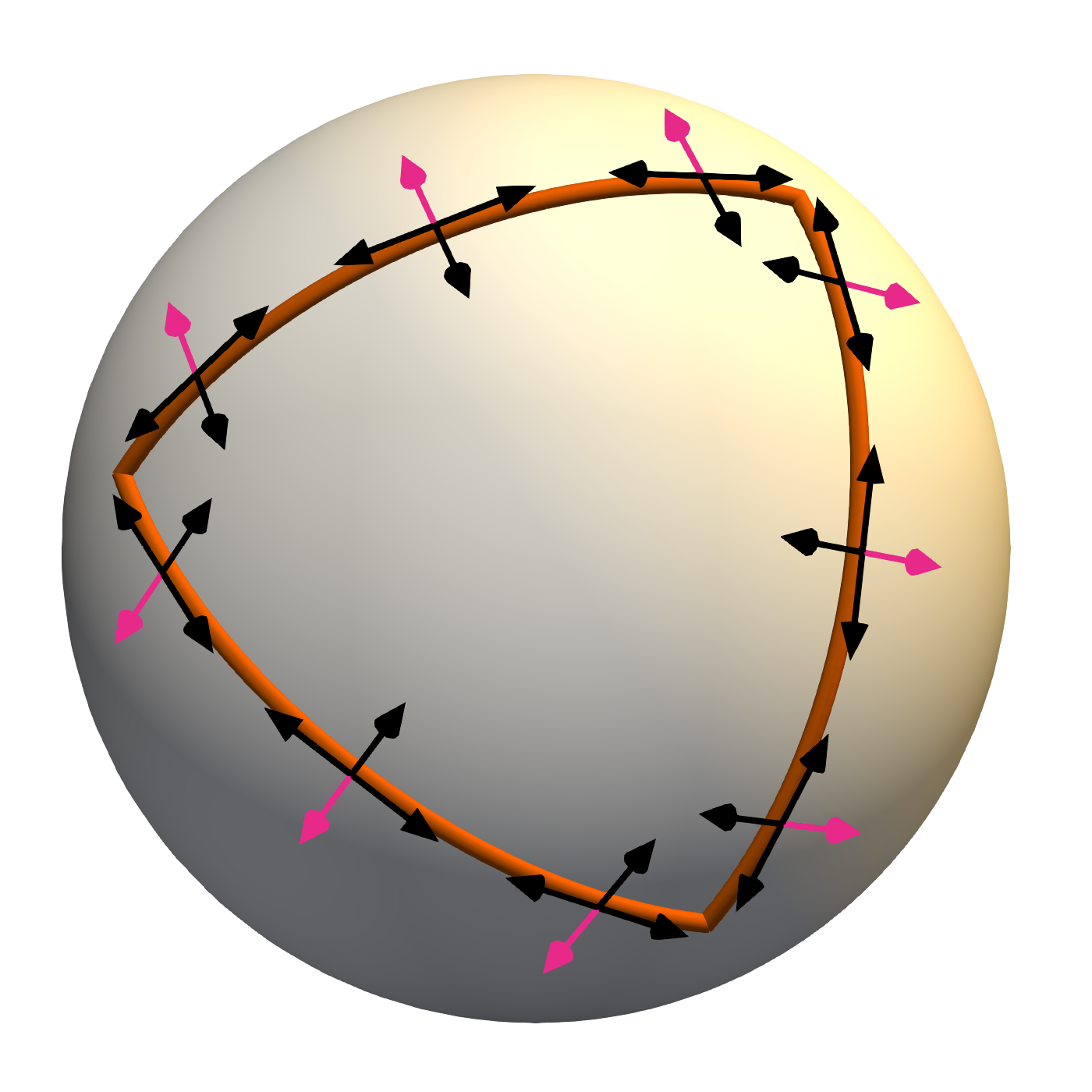}
        \subcaption{$\Pt^\gamma$, $\mu_4=1$}
    \end{subsubfigure}%
    \end{subfigure}%
    \begin{subfigure}[t]{0.5\textwidth}
    \hfill
    \setcounter{subsubfigure}{0}%
    \setcounter{subfigure}{1}%
    \begin{subsubfigure}{0.48\textwidth}
        \includegraphics[height=3.1cm]{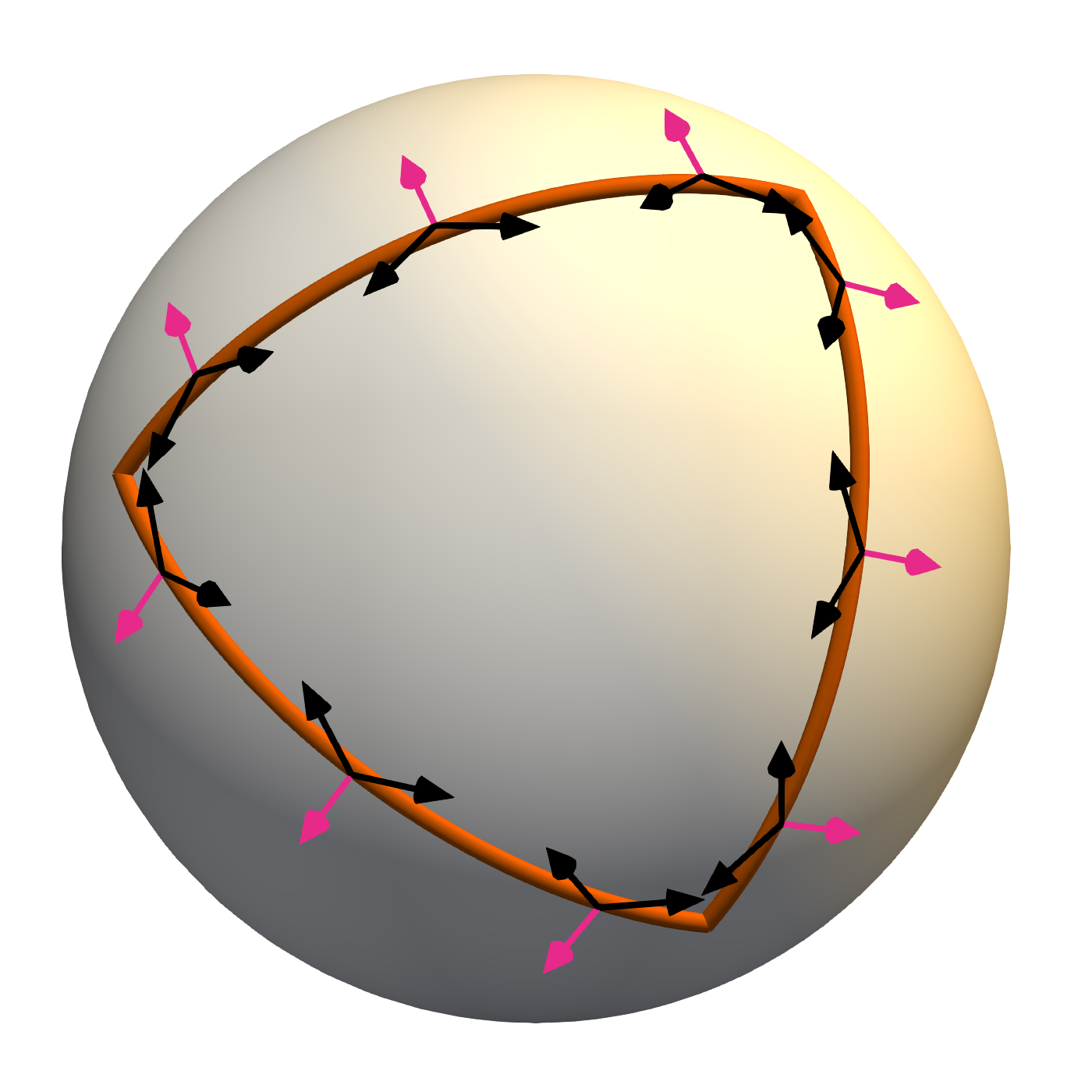}
        \subcaption{$\tPt^\gamma$, $\mu_3=1$}
    \end{subsubfigure}%
    \begin{subsubfigure}{0.48\textwidth}
        \includegraphics[height=3.1cm]{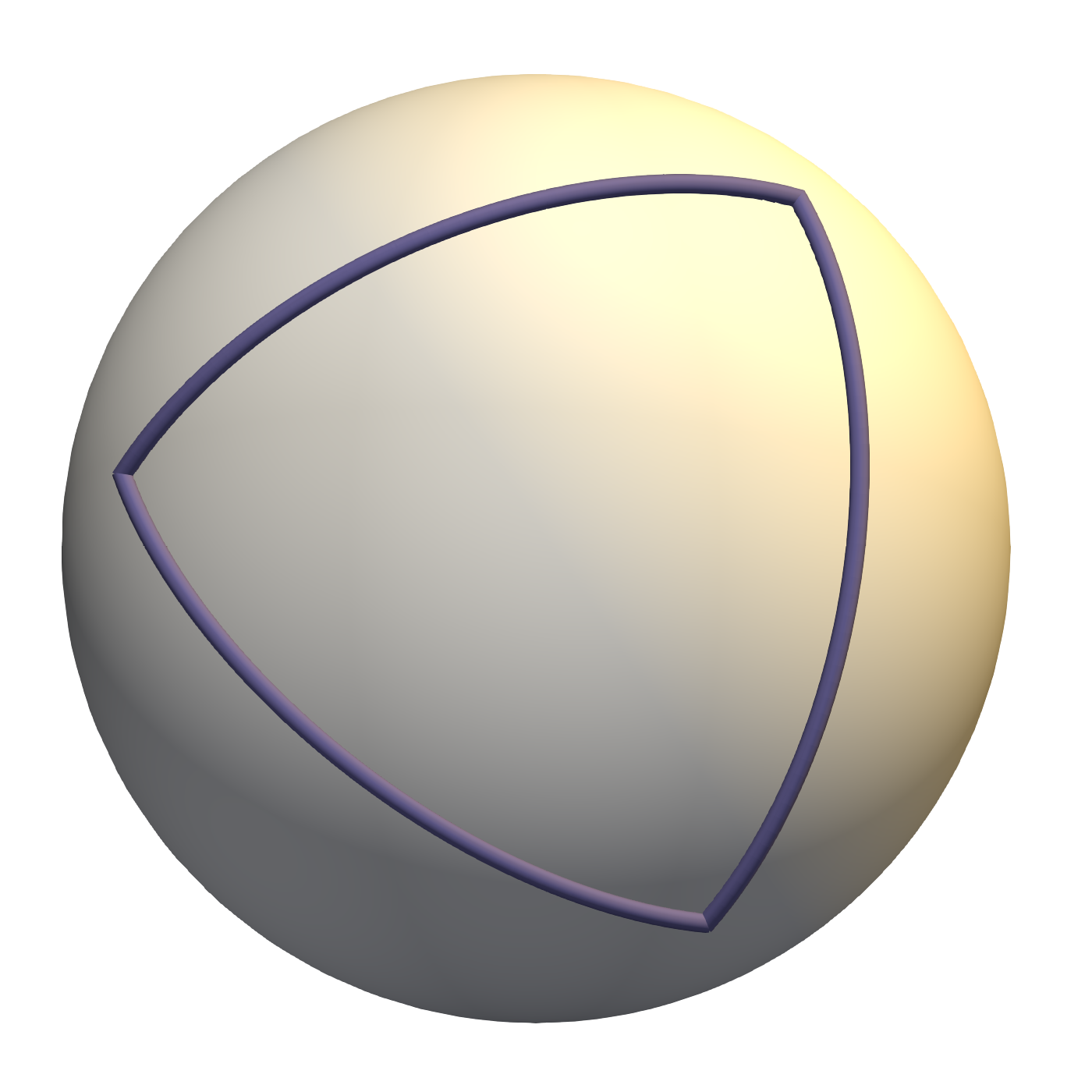}
        \subcaption{$\tPt^\gamma$, $\mu_4=0$}
    \end{subsubfigure}%
    \end{subfigure}
    \caption{\label{fig:minkowski-eigenvectors} Comparison of the ISMTs based on parallel transport (a) and based on the turned transport (b). The triangles are colored by $\mu$ (purple: $\mu=0$, green: $0<\mu<1$ and orange: $\mu=1$), $\conormal$ is always shown in pink. $\{\eb_n^{+}\}$ are shown in black. As there is no preferred direction for $\mu=0$ no arrows are shown in (b.ii).}
\end{figure}

In \Cref{fig:minkowski-eigenvectors} we compare the normalized eigenvalues and eigenvectors of the irreducible surface MT for a regular triangle for $p=3,4$. We expect the normalized eigenvalues $\mu_3=1$ and $\mu_4=0$. This would characterize a triangle. Indeed, in the right figures (b.i) and (b.ii) this can be observed. Opposed to this, we have also computed the normalized eigenvalues when transporting the independent components $g_j$ using the classical parallel transport $\Pt^\gamma$, as proposed in~\cite{Collischon2024Morphometry}. We neither find the $\mu_3 \approx 1.0$ nor the vanishing value $\mu_4$. On the contrary, the normalized eigenvalue $\mu_4=1$ indicates that the shape is more like a quadrilateral. This is, because the spherical triangle is constructed with three 90 degree inner angles, leading to the false shape characterization.
In addition to the normalized eigenvalues, we visualizes the associated eigenvectors along the curve. For $p=3$ with $\mu_3=1$ in (b.i) we can actually compute the corresponding eigenvectors, whereas for $p=4$ with $\mu_4=0$ in (b.ii) there is no well defined direction. Also for the non-turned (classical) parallel transport we show the directions in (a.i) and (a.ii). They deviate from the directions in (b.i). The non-turned transport leads to non-vanishing eigenvalues for $p=4$ and thus also to eigenvectors visualized in (a.ii). This does not represent the shape of a triangle.
We conclude that our definition of the irreducible surface MT on a sphere leads for regular polygons to the same properties as the well-known MT in flat space.

\begin{figure}
\centering
    \includegraphics[width=\textwidth]{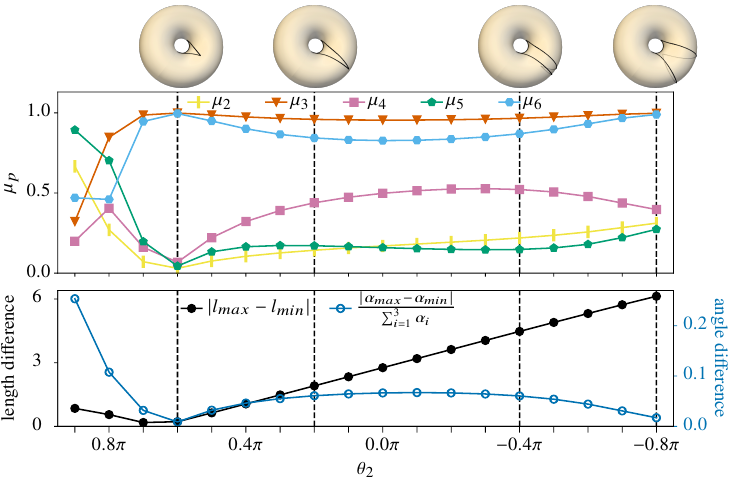}
    \caption{\label{fig:sketch-Torus} Numerical experiments with geodesic triangles on a torus.\ (top) Values of $\mu_p$ for the different triangles.\ (bottom) Difference between the smallest and the largest transport angle of the triangle (scaled with the sum of all angles) in blue, marked with an open circle, and the difference between the longest and the shortest edge of the triangle in black, marked with a filled circle. The dashed lines art for parameter $\theta_2=0.6\pi,0.2\pi,-0.4\pi,-0.8\pi$, from left to right.}
\end{figure}

While regular geodesic polygons exist on the sphere due to its constant curvature, this is not necessarily true on other surfaces. In order to analyze the effect of deviation from a perfect regular polygon on the normalized eigenvalues we next consider geodesic triangles on a torus. The setup is as follows: two vertices of the triangle are placed on the innermost ring of the torus with a geodesic distance of nearly $\pi r_0$, where $r_0=R-r$ is the inner radius, the third vertex is placed in the middle between these first two vertices, but then moved over the ring from a distance of nearly 0 to nearly $2\pi r$, see \Cref{fig:sketch-Torus}. We consider a torus parametrization
\[
  (x,y,z)=\big( (R+r\cos\theta)\cos\phi, (R+r\cos\theta)\sin\phi, r\sin\theta \big),\quad\text{ with }\phi,\theta \in [0,2\pi]
\]
for $r=1.375, R=2$. The vertices are then given in angles: $(\phi_{1},\theta_{1})=(0,\pi)$, $\phi_{2}=1.55$, $\theta_{2} \in [0.9\pi,-0.8\pi]$, and $(\phi_{3},\theta_{3})=(3.1,\pi)$, where the numbering of the vertices indicates the curve orientation.
In \Cref{fig:sketch-Torus} we have visualized the normalized eigenvalues $\mu_p$, $p=2\ldots,6$ for various triangular shapes. In the top plot we observe that $\mu_3$ and $\mu_6$ mostly obtain values greater than $0.8$, whereas $\mu_2,\mu_4$, and $\mu_5$ are much smaller in magnitude. In particular, the eigenvalue $\mu_3\approx 1$ for all angles $\theta_2 < 0.7\pi$. This allows to classify the shape as a triangle. We have two special triangular shapes, visualized along the first and last vertical dashed line, with $\theta_2\in\{0.6\pi, -0.8\pi\}$. In both cases, all angles in the triangle are nearly equal, i.e., the angle difference as visualized in the bottom plot is close to 0. This results in $\mu_3 \approx 1$ and $\mu_6\approx 1$. Only for $\theta_2=0.6\pi$, the first dashed line, also the triangle edges are nearly of equal length. In that case the other $\mu_m\approx 0, m=2,4,5$, whereas in the case $\theta_2=-0.8\pi$, the last dashed line, the edge lengths are quiet different and thus, the $\mu_m > 0, m=2,4,5$. This essentially shows that the properties of a triangle, characterized by the normalized eigenvalues $\mu_p$ depend non-linearly on the surface $\M$.

\subsection{Convergence properties for approximation of the curve and the surface}
In \Cref{sec:approximation} we have studied geodesic and straight line segment approximations of smooth curves on surfaces. In order to investigate the approximation properties and convergence results, we define the parametrization of a curve, called ``flower curve'' in the following, and then distribute points along that curve. The curve is given in terms of a polar coordinate parametrization with radius and two angles, over a parameter interval: $(r_0,a,\omega)$ with center $(x_0,y_0)=(\pi/2,\pi/4)$, for $t\in\interval{0}{2\pi}$
\begin{align}
    r(t) &\colonequals r_0 - a\sin(\omega\,t), \label{eq:flower-curve-radius} \\
    \theta(t) &\colonequals x_0 + r(t)\cos(t), \\
    \varphi(t) &\colonequals y_0 + r(t)\sin(t).
    \label{eq:flower-curve-phi}
\end{align}
The curve parametrization has to be understood in the parameter domain of a surface. We consider an ellipsoid parametrization with three major axes $a_1,a_2,a_3$, and denote the surface by $E(a_0,a_1,a_2)$. The curve is then given by
\begin{align}
    \gamma(t) &\colonequals \big(a_1\sin\theta(t)\cos\varphi(t), a_2\sin\theta(t)\sin\varphi(t), a_3\cos\theta(t)\big), & \text{for }t\in\interval{0}{2\pi}.
\end{align}

We consider the ellipsoid surface $E(1.6,1.3,1)$ with an approximation of a flower curve with $\omega=3$ by (a) a geodesic line approximation and (b) a straight line approximation. The smooth curve is thereby subdivided into $q$ equal length segments where the endpoints define the corresponding polygonal approximations. Note that the curve and the resulting polygons are not convex and thus the turning angles need to be computed with a sign in relation to the surface orientation.

\begin{table}
\centering
\begin{tabular}{rccccccccc}
  \toprule
  $q$ & $|L-L_q|$ & {(eoc)} & $|\kappa-\kappa_q|$ & (eoc) & $|f-f_q|$ & (eoc) & $\norm{\overline{\gb}_p-\overline{\gb}_{p,q}}$ & (eoc) \\\midrule
  4 & 8.6e-01 & --- & 3.9e-01 & --- & 2.4e+00 & --- & 3.1e-02 & --- \\
  16 & 9.0e-02 & 2.15 & 3.5e-02 & 1.87 & 1.2e+00 & 0.60 & 9.0e-02 & 1.16 \\
  64 & 7.1e-03 & 1.99 & 2.2e-03 & 1.99 & 3.4e-01 & 0.74 & 9.3e-03 & 1.41 \\
  256 & 4.5e-04 & 2.00 & 1.4e-04 & 2.00 & 8.8e-02 & 1.00 & 2.0e-03 & 1.07 \\
  1024 & 2.8e-05 & 2.00 & 8.5e-06 & 2.00 & 2.2e-02 & 1.00 & 4.8e-04 & 1.02 \\
\bottomrule
\end{tabular}
\caption{Convergence data for the geodesic approximation of a flower curve with $(r_0,a,\omega)=(0.7,0.2,3)$ and irreducible surface MT of rank $p=3$ on an ellipsoid surface $E(1.6,1.3,1.0)$.}\label{tab:geodesic_approx_data}
\end{table}

In \Cref{tab:geodesic_approx_data} and \Cref{tab:line_approx_data} we have listed the geometric quantities on several refinement levels of the approximation with increasing number of segments $q$ and how they converge. At first, the approximation of the total length $L$ by $L_q$ is considered. As stated in \Cref{rem:length_consistency}, it converges with order 2. The total geodesic curvature $\kappa=\kappa(\gamma)=\int_\gamma k_g(s)\ds$ and its approximation $\kappa_q=\kappa(\gamma_q)$ also shows second order convergence, see \Cref{lem:geodesic_approx}. The last two quantities, $f$ and $\overline{\gb}_{p}$, are directly related to the irreducible surface MT and allow to extract the normalized eigenvalues $\mu_p$ and eigenvectors $\eb_{p,n}$. These quantities converge with order 1, as expected by \Cref{cor:f-fh,cor:g-gh}. Note that the independent components $\overline{\gb}_{p}$ and $\overline{\gb}_{p,q}$ are always normalized by the curve length $L$ and $L_q$, respectively.

\begin{table}
\centering
\begin{tabular}{rccccccc}
  \toprule
  $q$ & $|L-L_q|$ & {(eoc)} & $|\kappa-\kappa_q|$ & (eoc) & $\norm{\overline{\gb}_p-\overline{\gb}_{p,q}}$ & (eoc) \\\midrule
  4 & 1.3e+00 & --- & 6.3e-01 & --- & 2.0e-02 & --- \\
  16 & 1.4e-01 & 1.97 & 3.5e-02 & 2.14 & 3.3e-02 & 2.56 \\
  64 & 1.2e-02 & 1.96 & 1.9e-03 & 2.01 & 2.8e-03 & 1.68 \\
  256 & 7.4e-04 & 2.00 & 1.2e-04 & 2.00 & 3.7e-04 & 1.29 \\
  1024 & 4.6e-05 & 2.00 & 7.3e-06 & 2.00 & 8.3e-05 & 1.04 \\
\bottomrule
\end{tabular}
\caption{Convergence data for the line approximation of a flower curve with $(r_0,a,\omega)=(0.7,0.2,3)$ and irreducible surface MT of rank $p=3$ on an ellipsoid surface $E(1.6,1.3,1.0)$.}\label{tab:line_approx_data}
\end{table}

Finally, in \Cref{fig:triangulated-sphere} also the surface is approximated, using a triangulation of the sphere with all vertices orthogonally projected to the surface. The figure shows three refinement levels of subdivision of the triangles into four sub-triangles in each refinement step. The curve on the triangulated surface is given as the zero-levelset of the function $\rho$
\begin{align*}
  \rho(\xb) &\colonequals \sqrt{{\theta(\xb)}^2 + {\varphi(\xb)}^2} - r(t(\xb)), \\
  t(\xb) &\colonequals \operatorname{atan2}(\varphi(\xb),\theta(\xb)) + \pi, \\
  \theta(\xb) &\colonequals \operatorname{acos}(x_3/\norm{\xb}) - \pi/2, \\
  \varphi(\xb) &\colonequals \operatorname{atan2}(x_2,x_1) - \pi/4,
\end{align*}
with the point $\xb=(x_1,x_2,x_3)\in\M_h$ in the triangulation. This represents a flower curve by reconstructing the spherical coordinates from a point in its neighborhood and then using the radius definition from~\eqref{eq:flower-curve-radius}. As an approximation of the curve, we use the line segments obtained by intersecting the zero-levelset of a piecewise linear interpolation of $\rho$ on $\M_h$ with the triangulation edges. Note that this results in a polygonal chain with all segments of different length. The triangulation and extraction of the zero-levelset is implemented in Dune~\cite{BastianEtAl2021Dune}, using Dune-FoamGrid~\cite{dune-foamgrid} for the grid representation and TPMC~\cite{Engwer2017TPMC} for the interface extraction.

\begin{figure}
    \centering
    \begin{subfigure}[c]{0.3\textwidth}
        \includegraphics[width=0.95\textwidth]{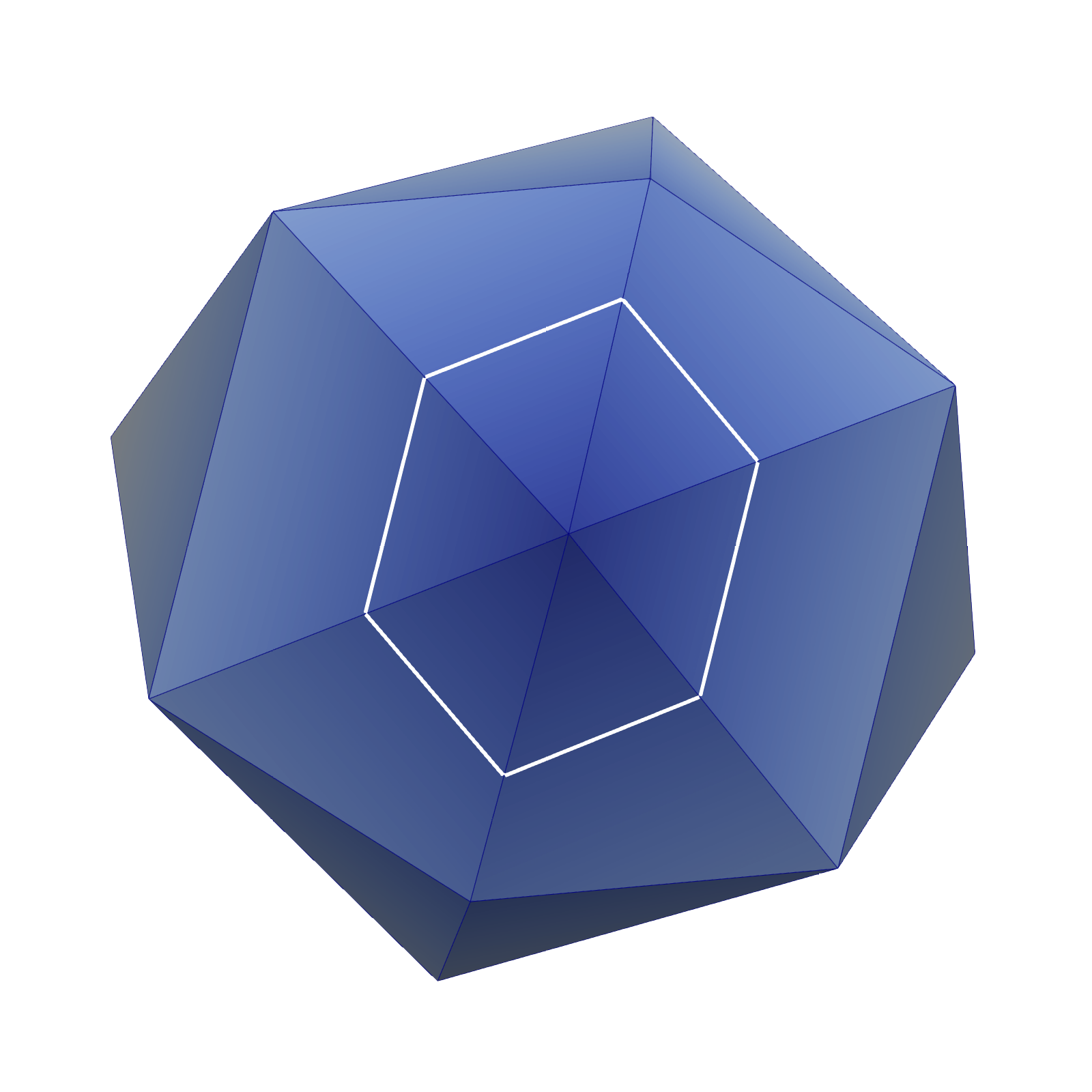}
        \subcaption{Level 0, $h=1$.}
    \end{subfigure}%
    \begin{subfigure}[c]{0.3\textwidth}
        \includegraphics[width=0.95\textwidth]{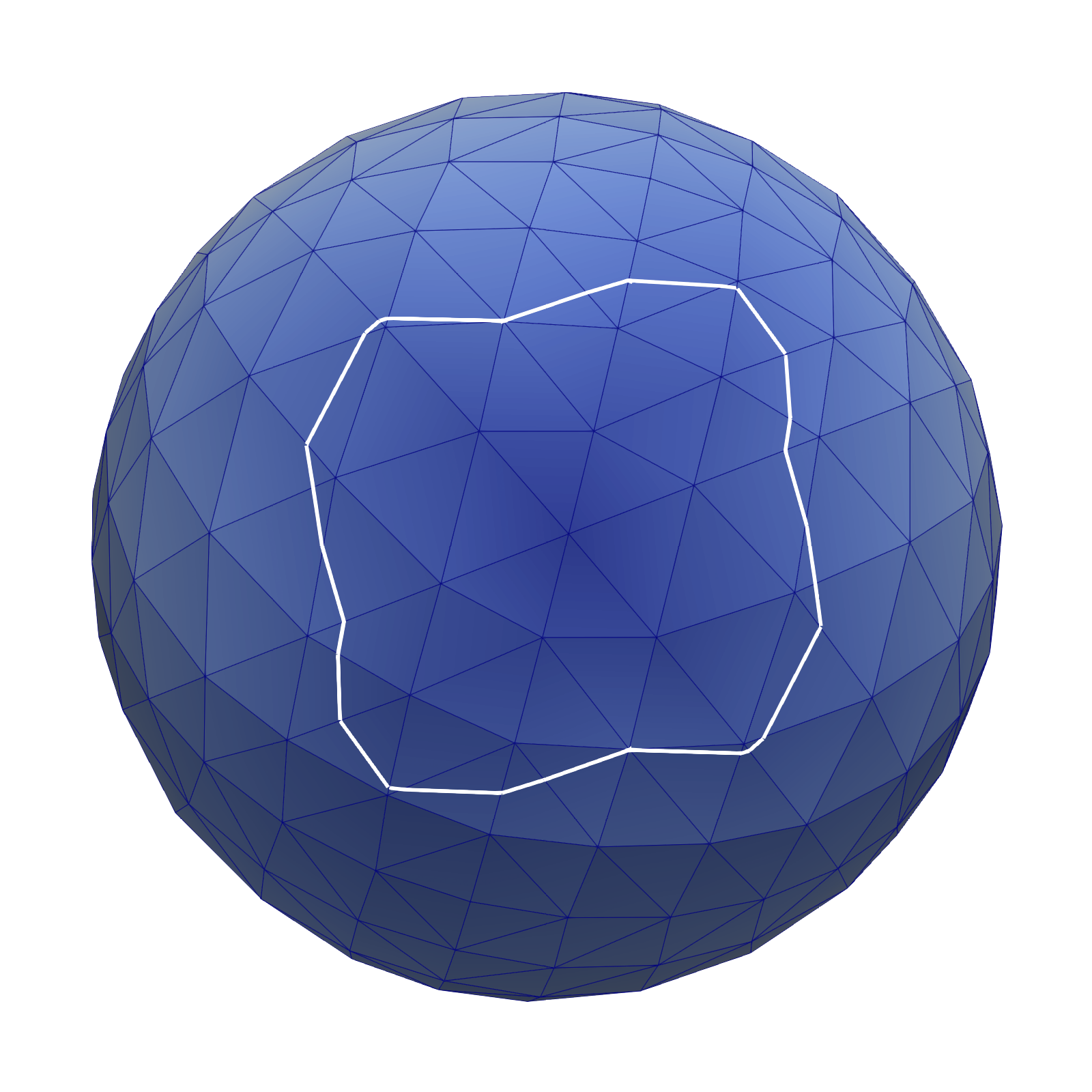}
        \subcaption{Level 2, $h=0.25$.}
    \end{subfigure}%
    \begin{subfigure}[c]{0.3\textwidth}
        \includegraphics[width=0.95\textwidth]{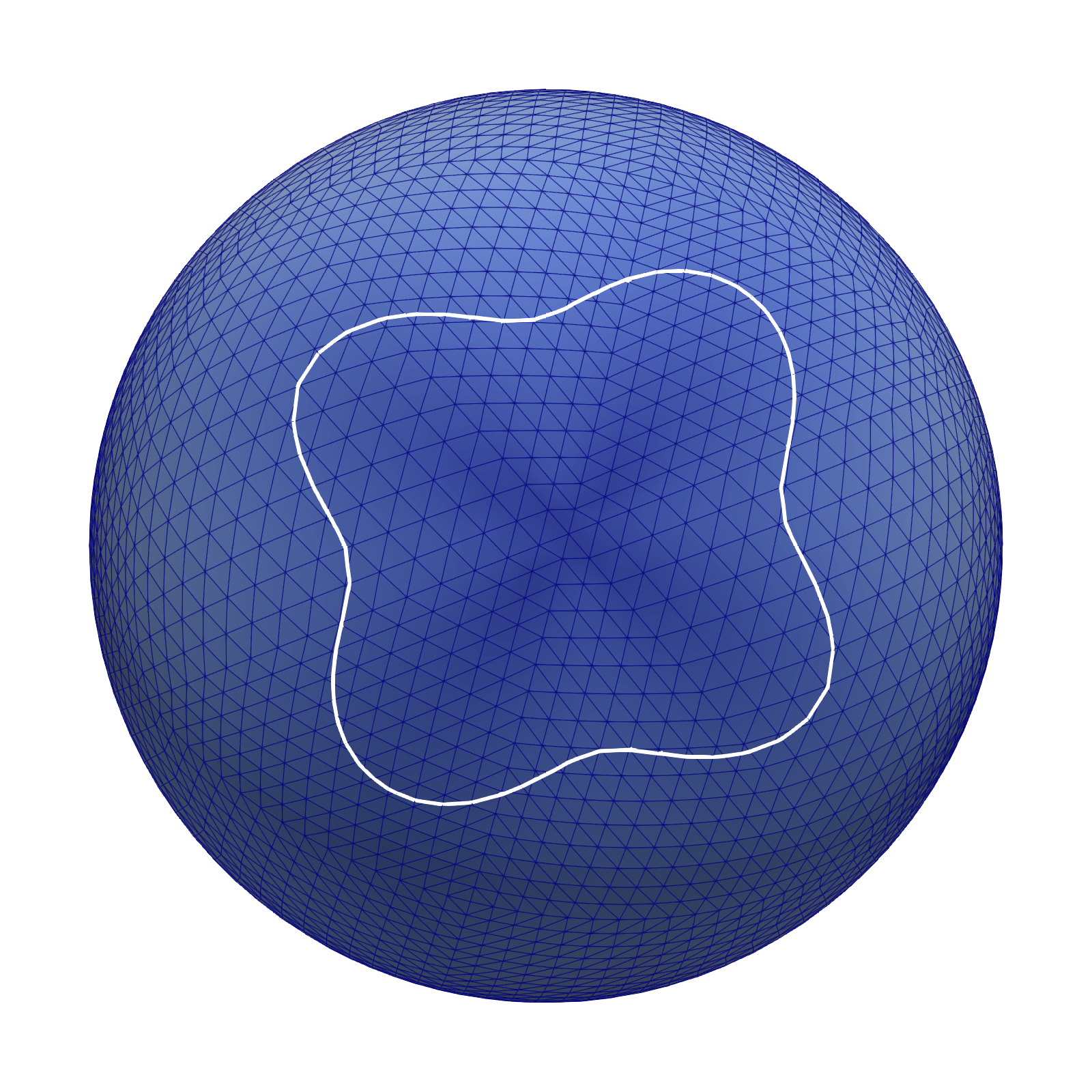}
        \subcaption{Level 4, $h=0.0625$.}
    \end{subfigure}%
    \caption{\label{fig:triangulated-sphere} Flower curve ($r_0=0.5, a=0.1, \omega=4$) projected and interpolated onto a spherical surface with decreasing grid size $h$. The line segments are obtained from intersection of the zero-levelset with the grid edges.}
\end{figure}

\begin{figure}
    \centering
    \includegraphics{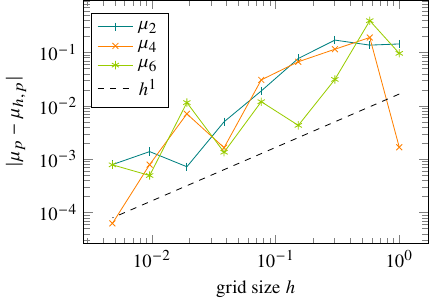}
    \caption{\label{fig:sphere-convergence}Errors in the computed normalized eigenvalues $\mu_{h,p}$ computed on the triangulated sphere $\M_h$ and compared against the values $\mu_p$ on the exact sphere $\M$. The dashed line indicates a linear convergence. The grid size $h$ refers to the maximal edge length in the sphere triangulation.}
\end{figure}

In \Cref{fig:sphere-convergence} the convergence of the normalized eigenvalues $\mu_{h,p}$, $h\to 0$ is studied. We have listed the values $\mu_{h,2},\mu_{h,4}$, and $\mu_{h,6}$ and their errors when compared against the exact surface and curve representation for various refinement levels. While we see that all $\mu_{h,p}$ seem to converge against $\mu_p$, the convergence graph is not a monotone function. This might be due to the fact that in each refinement step, the curve on the triangulated sphere must be constructed again by cutting the zero-levelset of the implicit function with the edges of the grid, resulting in a very non-homogeneous line segment distribution. Still, we can see an average linear convergence behavior in \Cref{fig:sphere-convergence}, as stated in \Cref{cor:linear-convergence-line-segment-approx}. The surface normal $\nb_h$ is extracted from the triangulation elements, averaged on the vertices and then assigned to the edge cut points.

\begin{figure}
    \begin{subfigure}[b]{0.475\textwidth}
            \centering
            \includegraphics[width=\linewidth]{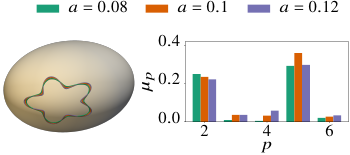}
            \subcaption{Modulation of the amplitude $a$, with fixed \mbox{$r_0=0.4$}, $\omega=5$, $y_0=\frac14\pi$}
    \end{subfigure}
    \hfill
    \begin{subfigure}[b]{0.475\textwidth}
            \centering
            \includegraphics[width=\linewidth]{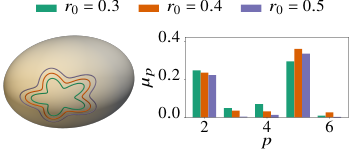}
            \subcaption{Modulation of the mean radius $r_0$, with fixed \mbox{$a=0.1$}, $\omega=5$, $y_0=\frac14\pi$}
    \end{subfigure}
    \vskip\baselineskip%
    \begin{subfigure}[b]{0.475\textwidth}
            \centering
            \includegraphics[width=\linewidth]{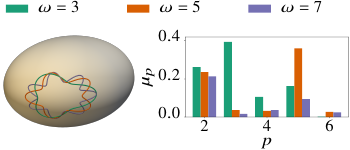}
            \subcaption{Modulation of the frequency $\omega$, with fixed \mbox{$r_0=0.4$}, $a=0.1$,$y_0=\frac14\pi$}
    \end{subfigure}
    \hfill
    \begin{subfigure}[b]{0.475\textwidth}
            \centering
            \includegraphics[width=\linewidth]{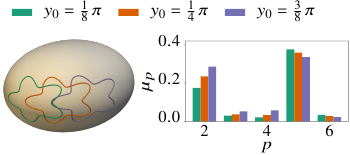}
            \subcaption{Modulation of the position $y_0$, with fixed \mbox{$r_0=0.4$}, $a=0.1$, $\omega=5$}
    \end{subfigure}
    \caption{\label{fig:sketch-Ellipsoid} Numerical experiments of varying parameters of the flower curve on an ellipsoid $E(1.6,1.3,1.0)$ and the normalized eigenvalues $\mu_p$ for $p=2\ldots10$.}
\end{figure}
\subsection{Applications}
Before we come to an application in biology in which the
characterization of cells in a curved monolayer of cells is considered, let's review current findings for flat monolayers of cells. In~\cite{happel2025quantifying} MT have been applied to segmented cells in Madin-Darby canine kidney (MDCK) monolayers and computational approaches for active vertex models and multiphase field models to model such tissues. These cells are mostly irregular and one essential finding is that $\mu_p$ captures fundamentally distinct aspects of such cell shapes being independent of each other for different $p$. Transforming this to cells in curved monolayers of cells requires to not focus on one specific $p$ but to compute $\mu_p$ for various $p$ for each cell. We first test this by varying parameters of the flower curve defined in eqs.~\eqref{eq:flower-curve-radius}--\eqref{eq:flower-curve-phi}. The various curves and the corresponding values for $\mu_p$ for $p = 2,3,4,5,6$ are shown in \Cref{fig:sketch-Ellipsoid}. While mostly $\mu_2$ and $\mu_5$ obtain the highest values in the considered parameter regime, several values of $\mu_p$ strongly vary, leading, e.g., to larger values for $\mu_3$ for $r_0 = 0.4$, $a = 0.1$, $y_0 = \frac{1}{4} \pi$ and $\omega = 3$, to lower values for $\mu_5$ for $r_0 = 0.4$, $a = 0.1$, $y_0 = \frac{1}{4} \pi$ and $\omega = 3$ and $\omega = 7$ in (c), or to decreased values of $\mu_2$ in (d) if the curve is placed towards regions of larger curvature of the underlying surface $\mathcal{M}$ by moving $y_0 = \frac{3}{8} \pi$ to $y_0 = \frac{1}{4} \pi$ and $y_0 = \frac{1}{8} \pi$. These results suggest that $p$ should not be restricted to a specific number apriori, and characterization of rotational symmetries of shapes embedded in curved surfaces presumably requires to consider all $p$.

We follow this and analyze data from Arapidospis thaliana flower buds from~\cite{Li2019-ak} which has been used to demonstrate the applicability of the \href{https://nikolas-claussen.github.io/blender-tissue-cartography/}{blender-tissue-cartography} tool introduced in~\cite{Claussen2025_Blender} and provided by Nikolas Claussen with the permission to use the data in this publication. The tool extracts and cartographically projects surfaces from volumetric biological image data. In \Cref{fig:sketch-blender} a 3D rendering of the curved monolayer of cells (Fig. 3 $B^{\prime\prime}$ in~\cite{Claussen2025_Blender}) is shown together with segmented cells and three selected ones, which are analyzed using irreducible surface MT.\@ The cell contours (green) together with the surface normals (purple) are shown and the values for $\mu_p$ are provided for $p=2,3,4,5,6$, showing again a strong variability of the values underlining that no single $p$ can capture the full information. To get rid of pixel shaped artifacts the cell contours where smoothened by replacing every coordinate with the mean of three coordinates.  In addition we also compute the eigenvectors that can give an indicator for a preferred direction. The direction fields are plotted in the local tangent plane of $\mathcal{M}$ and are shown for $p=2,3,4,5,6$. The coordinate system is rotated for each cell to highlight the change in the surface normal and the direction fields, respectively. We here only demonstrates the ability to extract the information from microscopy data and to compute irreducible surface MT for such data. How this information can be linked to cellular behaviors and tissue organization is not topic of this article. We therefore only sketch open problems emerging by proceeding in this direction. In contrast to flat monolayers of cells, where the computed eigenvectors would be constant for each cell and can be used to obtain a coarse-grained $p$-atic field that characterizes the orientational order at the tissue scale and can be related to $p$-atic liquid crystal theories and thus mechanical properties~\cite{Giomi_PRE_2022}, such approaches are not available for curved monolayers of cells.

\begin{figure}[h!]
    \begin{subfigure}[b]{0.27\textwidth}
            \centering
            \includegraphics[width=\linewidth]{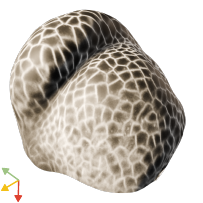}
            \subcaption{}
    \end{subfigure}
    \hfill
    \begin{subfigure}[b]{0.29\textwidth}
            \centering
            \includegraphics[width=\linewidth]{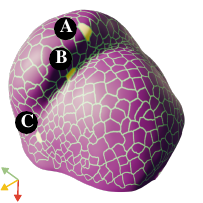}
            \subcaption{}
    \end{subfigure}
    \hfill
    \begin{subfigure}[b]{0.386\textwidth}
            \includegraphics[width=\linewidth]{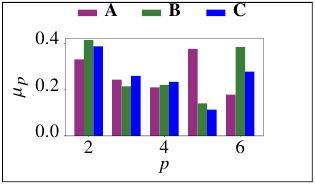}
            \subcaption{}
    \end{subfigure}
     \vskip\baselineskip%
    \begin{subfigure}[c]{0.93\textwidth}
            \centering
            \includegraphics[width=\linewidth]{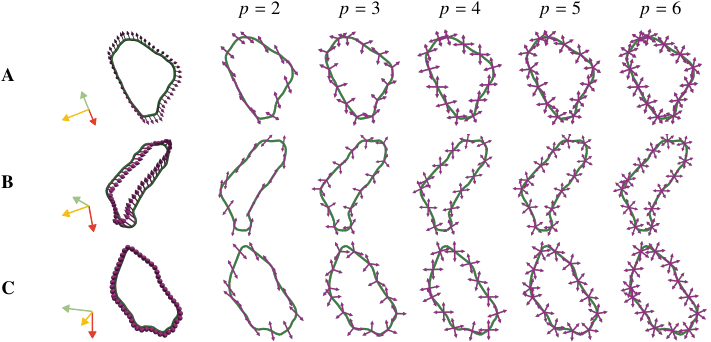}
            \subcaption{}
    \end{subfigure}
    \caption{\label{fig:sketch-blender} $\mu_p$ and $\vartheta^-_p$ for experimental data. $\mathbf{(a)}$ shows the experimental image. In $\mathbf{(b)}$ the segmentation of $\mathbf{(a)}$ is shown and the cells used in the following pictures are marked. $\mathbf{(c)}$ $\mu_p$ for the three chosen cells is shown. In $\mathbf{(d)}$ we show the outline of the cells in green and the direction of the surface normal and $\vartheta^-_p$ in purple (from left to right). As $\vartheta^-_p$ is in the tangent space of the respective point the arrows appear more flat there. In $(\mathbf{a})$, $(\mathbf{b})$ and $(\mathbf{d})$ the $X$ (red), $Y$ (yellow) and $Z$ (green) -axis are shown.}
\end{figure}

\section*{Acknowledgments}
We thank Nikolas Claussen for providing the raw data used in \Cref{fig:sketch-blender}.  Furthermore we thank him for providing \href{https://nikolas-claussen.github.io/blender-tissue-cartography/}{blender-tissue-cartography} (a tissue cartography add-on for Blender with an accompanying Python library) and explaining the usage of said library to us.

\section*{Funding}
This work was supported by the Deutsche Forschungsgemeinschaft (DFG, German Research Foundation) through the research unit FOR~3013, ``Vector- and Tensor-Valued Surface PDEs'' (project number 417223351), within the projects TP05, ``Ordering and defects on deformable surfaces'', and TP06, ``Symmetry, length, and tangential constraints''.

\section*{Authors and affiliations}
\begin{description}
\item[Lea Happel \orcid{0000-0002-4525-9185}]
  Institute of Scientific Computing, Technische Universität Dresden, 01062 Dresden, Germany;\\
  \href{mailto:lea.happel@tu-dresden.de}{lea.happel@tu-dresden.de}

\item[Hanne Hardering \orcid{0009-0002-8173-7807}]
  Institute of Numerical Mathematics, Technische Universität Dresden, 01062 Dresden, Germany;\\
  \href{mailto:hanne.hardering@tu-dresden.de}{hanne.hardering@tu-dresden.de}

\item[Simon Praetorius \orcid{0000-0002-1372-4708}]
  Institute of Scientific Computing, Technische Universität Dresden, 01062 Dresden, Germany;\\
  \href{mailto:simon.praetorius@tu-dresden.de}{simon.praetorius@tu-dresden.de}

  \item[Axel Voigt \orcid{0000-0003-2564-3697}]
  Institute of Scientific Computing, Technische Universität Dresden, 01062 Dresden, Germany;
  Cluster of Excellence Physics of Life (PoL), Technische Universität Dresden, 01062 Dresden, Germany;
  Center for Systems Biology Dresden (CSBD), Pfotenhauerstraße 108, 01307 Dresden, Germany;\\
  \href{mailto:axel.voigt@tu-dresden.de}{axel.voigt@tu-dresden.de}

\end{description}

\bibliographystyle{abbrv}
\bibliography{references}

\end{document}